\documentclass[12pt, twoside]{article}
\usepackage{lipsum}
\usepackage{tocloft}

\usepackage[all, cmtip]{xy}

\usepackage[colorlinks=true]{hyperref}

\usepackage{amsmath,amsthm,amssymb}
\usepackage{times}

\usepackage{times,upref,eucal, mathrsfs, xcolor, dsfont}
\usepackage{amsfonts,amsmath,amstext,amsbsy,amssymb, amsopn,amsthm}

\usepackage{enumerate}

\usepackage{url}
\usepackage{relsize}

\usepackage{mathtools}
\usepackage{bookmark}

\theoremstyle{plain}
\newtheorem{thm}{Theorem}[section]
\newtheorem{lem}[thm]{Lemma}
\newtheorem{prop}[thm]{Proposition}
\newtheorem{cor}[thm]{Corollary}

\newtheorem*{problem}{Problem}

\theoremstyle{definition}
\newtheorem{defn}{Definition}[section]

\newtheorem{con}{Condition}

\theoremstyle{remark}
\newtheorem{rem}{Remark}[section]

\newtheorem*{notation}{Notation}

\newcommand{\thistheoremname}{}
\newtheorem{genericthm}[thm]{\thistheoremname}

  \newtheorem*{genericthm*}{\thistheoremname}
\newenvironment{namedthm*}[1]
  {\renewcommand{\thistheoremname}{#1}%
   \begin{genericthm*}}
  {\end{genericthm*}}

\newcommand{\N}{\mathbb{N}}
\newcommand{\R}{\mathbb{R}}
\newcommand{\C}{\mathbb{C}}
\newcommand{\Z}{\mathbb{Z}}

\newcommand{\D}{\mathbb{D}}

\newcommand{\E}{\mathbb{E}}
\newcommand{\PP}{\mathbb{P}}

\newcommand{\Prob}{\mathbb{P}}

\newcommand{\supp}{\mathrm{supp}}
\newcommand{\Var}{\mathrm{Var}}

\newcommand{\sgn}{\mathrm{sgn}}
\newcommand{\Ran}{\mathrm{Ran}}

\newcommand \tr {{\mathrm{tr}}}

\newcommand \Conf {{\mathrm {Conf}}}

\newcommand \an { \text{\,\,and\,\,}}

\newcommand \new {\mathrm{new}}

\input xy
\xyoption{all}

\def\titlerunning#1{\gdef\titrun{#1}}
\makeatletter
\def\author#1{\gdef\autrun{\def\and{\unskip, }#1}\gdef\@author{#1}}
\def\address#1{{\def\and{\\\hspace*{18pt}}\renewcommand{\thefootnote}{}%
\footnote {#1}}%
\markboth{\autrun}{\titrun}}
\makeatother
\def\email#1{e-mail: #1}
\def\subjclass#1{{\renewcommand{\thefootnote}{}%
\footnote{\emph{Mathematics Subject Classification (2010):} #1}}}
\def\keywords#1{\par\medskip
\noindent\textbf{Keywords.} #1}

\begin{document}

\baselineskip=17pt

\titlerunning{\texorpdfstring{$J$}{a}-Hermitian determinantal point processes}

\title{\texorpdfstring{$J$}{a}-Hermitian determinantal point processes: balanced rigidity and balanced Palm equivalence}

\author{Alexander I. Bufetov, Yanqi Qiu}

\date{}

\maketitle

\address{Alexander I. Bufetov: Aix-Marseille Universit{\'e}, Centrale Marseille, CNRS, I2M, UMR7373,  39 Rue F. Juliot Curie 13453, Marseille; Steklov Institute of Mathematics, Moscow; Institute for Information Transmission Problems, Moscow; National Research University Higher School of Economics, Moscow; \email{bufetov@mi.ras.ru}}

\address{Yanqi Qiu: CNRS, Institut de Math{\'e}matiques de Toulouse, Universit{\'e} Paul Sabatier, 118, Route de Narbonne, F-31062 Toulouse Cedex 9; \email{yqi.qiu@gmail.com}}

\subjclass{Primary 60G55; Secondary 30H20, 30B20}

\begin{abstract}
We study  Palm measures of determinantal point processes with $J$-Hermitian correlation kernels. A point process $\Prob$  on the punctured real line $\R^* = \R_{+} \sqcup \R_{-}$ is said to be {\it balanced rigid} if for any precompact subset $B\subset \R^*$, the {\it difference} between the numbers of particles of a configuration inside $B\cap \R_+$ and $B\cap \R_-$ is almost surely determined by 
the configuration outside  $B$. The point process $\Prob$ is said to have the {\it balanced Palm equivalence property} if any reduced Palm measure conditioned at  $2n$ distinct points, $n$ in $\R_+$ and $n$ in $\R_-$,  is equivalent to the $\Prob$. 

We formulate general criteria for  determinantal point processes  with $J$-Hermitian correlation kernels to be balanced rigid  and to have the balanced Palm equivalence property and prove, in particular, that the determinantal point processes with Whittaker kernels of Borodin and Olshanski are balanced rigid and have the balanced Palm equivalence property. 

\keywords{ Determinantal point processes; $J$-Hermitian kernel; Whittaker kernels; $L$-processes; Palm measures; balanced rigidity;  balanced Palm equivalence property.}
\end{abstract}

\input xy
\xyoption{all}

\tableofcontents

\section{Introduction}
\subsection{Palm measures of determinantal point processes}

The present paper is the first one devoted to the equivalence and mutual singularity relations between reduced Palm measures of determinantal point processes with $J$-Hermitian correlation kernels.

A a concrete model, we consider the family of determinantal point processes  on the punctured real line $\R^* = \R \setminus \{0\}$ with Whittaker kernels of Borodin and Olshanski \cite{PII, PV}, scaling limits of the so-called $z$-measures of partitions \cite{BO-hyper, BO-R}.  For these determinantal point processes, we observe a new effect: the reduced Palm measure conditioned at $2n$ points, $n$ on the positive, $n$ on the negative semi-axis, is equivalent to the initial determinantal measure; while if $n \ne k$, then the initial measure and the reduced Palm measure conditioned at $n+k$ points, $n$ on the positive, $k$ on the negative semi-axis,  are mutually singular. In the former case, the Radon-Nikodym derivatives between the  reduced Palm measures and the initial determinantal measure are found explicitly as {\it regularized multiplicative functionals}.

In the case of determinantal measures with kernels given by Hermitian projection operators, the statement that two such measures differ by a multiplicative functional can be checked on the level of the corresponding subspaces, the ranges of our projections: in fact, it suffices to verify that these subspaces differ by multiplication by a function, see \cite{Buf-multi, Buf-inf} for precise statements.

Although $J$-Hermitian operators considered in this paper are  closely related to certain Hermitian projection operators, it does not seem possible to work with their ranges. Instead, we use the fact that the determinantal point processes with the Whittaker kernels admit so-called $L$-kernels. Following Borodin and Olshanski, such processes will be called $L$-processes. Two $L$-processes differ by a multiplicative functional once corresponding $L$-kernels themselves differ by multiplication by a function on the left and on the right. 

The realization of this scheme requires some effort. First, Palm measures of an $L$-process, generally speaking, do not admit an $L$-kernel (this can be seen already on the level of discrete phase spaces: indeed, Borodin and Rains \cite{Borodin-Rains} shown that any determinantal point process can be obtained from an $L$-process by conditioning).  Second, in developing the formalism of the regularized multiplicative functionals, we are not able to use the standard linear statistics as in \cite{BQS, QB3}.  We  use the {\it twisted} ones instead (see  \eqref{defn-linear-statistics} and \eqref{defn-twist-statistics} below for the definitions); in particular, an extended version of Fredholm determinants is used.

\subsection{Main results for Whittaker kernels}
We start by formulating our main results for a concrete model: the family of  the  determinantal point processes with Whittaker kernels of Borodin and Olshanski. The reader is referred to \cite{BO-hyper, PII, PV} for the origin of these point processes in the problem of harmonic analysis on the infinite symmetric group and to \S \ref{section-pre} below for a reminder of  the main definitions related to determinantal point processes. 

Let $\R^* = \R \setminus \{0\}$ be the punctured real line.  By a {\it configuration} on $\R^*$, we mean  a {\it locally finite} subset  $ X \subset \R^*$, that is, $X$ is a subset of $\R^*$ such that for any compact subset $B\subset \R^*$, the cardinality $\#(X \cap B)$ of the intersection of the subsets $X$ and $B$ is finite. Define the {\it space of configurations} on $\R^*$ by
$$
\Conf(\R^*): = \{X \subset \R^*: \text{\, for any compact subset $B\subset \R^*$, $\#(X \cap B)<\infty$}\}.
$$
The space of configurations $\Conf(\R^*)$ is naturally equipped with a Borel structure, see \S \ref{section-pre}. A point process on $\R^*$ is  by definition a Borel probability on $\Conf(\R^*)$.

 The family of  the  determinantal point processes with Whittaker kernels of Borodin and Olshanski is a 2-parameter family $\mathscr{P}_{z,z'}  \footnote{It was denoted as $\widetilde{\mathcal{P}}_{z, z'}$ in \cite{BO-hyper}.}$ of determinantal point processes on $\R^*$. The two parameters $z, z' \in \C$ satisfy one of the following conditions:
\begin{itemize}
\item either $z' = \bar{z}$ and $z \in \C\setminus \Z$,
\item or $z, z'\in \R$ and their exists $m \in \Z$ such that $m < z, z' < m +1$.
\end{itemize}
Following \cite[formula (5.6)]{BO-hyper}, we now write the correlation kernel of the determinantal point process $\mathscr{P}_{z, z'}$ explicitly.  Fix two parameters $z, z'\in \C$ such that one of the two conditions as above is satisfied. Set
\begin{align}\label{PQ}
\begin{split}
\mathcal{P}_{\pm} (x)  &=  \frac{(zz')^{1/4}}{ (\Gamma(1 \pm z ) \Gamma(1 \pm z') x)^{1/2}} W_{\frac{\pm (z+z') + 1}{2}, \frac{z-z'}{2}} (x),
\\
\mathcal{Q}_{\pm} (x) &=  \frac{(zz')^{3/4}}{ (\Gamma(1 \pm z ) \Gamma(1 \pm z') x)^{1/2}} W_{\frac{\pm (z+z') - 1}{2}, \frac{z-z'}{2}} (x),
\end{split}
\end{align}
where $\Gamma(\cdot)$ is the Euler Gamma-function and $W_{a, b}(\cdot)$ is the Whittaker function with parameter $a, b \in \C$, see \cite[6.9]{Erdelyi-vol1} for the definition of Whittaker functions.
The correlation kernel of  $\mathscr{P}_{z, z'}$ is given by 
\begin{align}\label{whittaker}
\begin{split}
 \mathcal{K}_{z,z'}(x,y) = 
\left\{ \begin{array}{ll}   
\mathlarger{\frac{\mathcal{P}_{+}(x) \mathcal{Q}_{+}(y) - \mathcal{Q}_{+}(x) \mathcal{P}_{+}(y) }{x-y}}  , &\text{\, for $x> 0, y > 0$; } \vspace{3mm}
\\
 \mathlarger{ \frac{\mathcal{P}_{+}(x) \mathcal{P}_{-}(-y) + \mathcal{Q}_{+}(x) \mathcal{Q}_{-}(-y) }{x-y} }, &\text{\, for $x> 0, y < 0$; }\vspace{3mm}
\\
  \mathlarger{ \frac{\mathcal{P}_{-}(-x) \mathcal{P}_{+}(y) + \mathcal{Q}_{-}(-x) \mathcal{Q}_{+}(y) }{x-y} }, &\text{\, for $x< 0, y > 0$; }\vspace{3mm}
\\
  \mathlarger{ \frac{\mathcal{P}_{-}(-x) \mathcal{Q}_{-}(-y) - \mathcal{Q}_{-}(-x) \mathcal{P}_{-}(-y) }{y-x}}, &\text{\, for $x< 0, y < 0$.}
\end{array}
 \right.
\end{split}
\end{align}
These kernels $\mathcal{K}_{z,z'}$ are called Whittaker kernels.

Recall that given  a finite set  $S$, we denote its cardinality by $\#(S)$.  Denote by $\R_{+}$ the positive semi-axis and $\R_{-}$ the negative semi-axis. Our first main result, in case of Whittaker kernel model, is

\begin{namedthm*}{Theorem A}
Assume that the parameters $(z,z')$ are such that  $z' = \bar{z}$ and $z \in \C \setminus \R$.  Then for any subset $B\subset\R^*$ having a positive distance from the origin,   the difference 
$$
\#(B \cap X \cap \R_{+}) -  \#(B \cap X \cap \R_{-})
$$ 
 is  $\mathscr{P}_{z,z'}$-almost surely determined by $X \cap (\R^* \setminus B)$, the configuration outside $B$. That is, there exists a measurable function $N_B^{out}: \Conf(\R^*) \rightarrow \Z$, such that for $\mathscr{P}_{z,z'}$-almost every configuration $X\in \Conf(\R*)$, we have
$$
\#(B \cap X \cap \R_{+}) -  \#(B \cap X \cap \R_{-})  = N^{out}_B(X\setminus B).
$$

 In particular, if $B \subset \R_{+}$ is a subset in the positive semi-axis with a positive distance from the origin, then $\#(B \cap X)$ is $\mathscr{P}_{z,z'}$-almost surely determined by $X \cap (\R^* \setminus B)$. If $B$ is in the negative semi-axis, the same result holds. 
 \end{namedthm*}

\begin{rem}
When the subset $B$ is either in positive semi-axis or in negative semi-axis, we recover the usual {\it number rigidity property}  of Ghosh \cite{Ghosh-sine}, Ghosh and Peres \cite{Ghosh-rigid}.
\end{rem}

If $\PP$ is a point process on $\R^*$ and if $\mathfrak{p} = (p_1, \dots, p_m) \in (\R^*)^m$ is an $m$-tuple of distinct points in $\R^*$, then we denote $\PP^{\mathfrak{p}}$ the {\it reduced} Palm measure of $\PP$ conditioned at the points $p_1, \dots, p_m$. See \S \ref{sec-Palm} for the formal definition of the reduced Palm measures. 

Using a variant of Proposition 8. 1  in \cite{QB3}, we derive from Theorem A  the following corollary.
\begin{cor}
Assume that the parameters $(z,z')$ are such that  $z' = \bar{z}$ and $z \in \C \setminus \R$.  Let $n, k$ be two non-negative integers such that $n\ne k$. Then for Lebesgue-almost every $\mathfrak{p} = (p_1^{+}, \dots, p_n^{+}; p_1^{-}, \dots, p_k^{-}) \in \R_{+}^n \times \R_{-}^k $ of distinct points,  the reduced Palm measure 
$\mathscr{P}_{z,z'}^{\mathfrak{p}}$ and the initial determinantal measure $\mathscr{P}_{z,z'}$ are mutually singular.
\end{cor}

We now proceed to formulating our second main result which gives equivalence of the reduced Palm measures $\mathscr{P}_{z,z'}^{\mathfrak{p}}$ and  $\mathscr{P}_{z,z'}$, under the conditions that the parameters $z,z'\in \C$ are such that $| z + z'| <1$ and $\mathfrak{p} = (p_1^{+}, \dots, p_n^{+}; p_1^{-}, \dots, p_n^{-}) \in \R_{+}^n \times \R_{-}^n$ is a $2n$-tuple of distinct points in $\R^*$ with equal numbers of points from positive and  negative semi-axis. The Radon-Nikodym derivative $d\mathscr{P}_{z,z'}^{\mathfrak{p}}/ d\mathscr{P}_{z,z'}$ is computed explicitly. 

We start with an auxiliary proposition

\begin{prop}\label{intro-prop-1}
Assume that the two parameters $z,z'\in \C$ are such that $| z + z'| <1$.  Then  the following limit
\begin{align*}
\overline{S}_{\mathfrak{p}}(X) : & =  \lim_{\delta\to 0^{+}}    \bigg\{   \sum_{i = 1}^n   \Big(\sum_{x \in X \cap (\delta, \infty)}\log \left| \frac{x/p_i^{+}  - 1 }{x/p_i^{-} - 1} \right|^2 - \sum_{x \in X\cap (-\infty, - \delta) } \log \left| \frac{x/p_i^{-}  - 1 }{x/p_i^{+} - 1} \right|^2\Big)  
\\
& - \E_{ \mathscr{P}_{z,z'}}  \sum_{i = 1}^n   \Big(\sum_{x \in X \cap (\delta, \infty)}\log \left| \frac{x/p_i^{+}  - 1 }{x/p_i^{-} - 1} \right|^2 - \sum_{x \in X\cap (-\infty, - \delta) } \log \left| \frac{x/p_i^{-}  - 1 }{x/p_i^{+} - 1} \right|^2\Big)
\bigg\}
\end{align*}
exists for $ \mathscr{P}_{z,z'}$-almost every configuration $X\in \Conf(\R^*)$.  Moreover, we have  
$$
 \exp ( \overline{S}_{\mathfrak{p}}) \in L^1(\Conf(\R^*),  \mathscr{P}_{z,z'}).
 $$
\end{prop}

\begin{namedthm*}{Theorem B}
Assume that the two parameters $z,z'\in \C$ are such that $| z + z'| <1$.  Then  the determinantal point process $\mathscr{P}_{z,z'}$  possesses the following property: for Lebesgue almost every $\mathfrak{p} = (p_1^{+}, \dots, p_n^{+}; p_1^{-}, \dots, p_n^{-}) \in \R_{+}^n \times \R_{-}^n$ of distinct points,  the reduced Palm measure 
$\mathscr{P}_{z,z'}^{\mathfrak{p}}$ is equivalent to the initial determinantal measure $\mathscr{P}_{z,z'}$. For the Radon-Nikodym derivative, we have the  $ \mathscr{P}_{z,z'}$-almost sure equality 
$$
\frac{d \mathscr{P}_{z,z'}^{ \mathfrak{p}}}{  d \mathscr{P}_{z,z'}} (X)=  \frac{\exp ( \overline{S}_{\mathfrak{p}}(X) )
}{\E_{\mathscr{P}_{z,z'}} \Big[\exp ( \overline{S}_{\mathfrak{p}} )
 \Big]}.
$$
\end{namedthm*}

\begin{rem}
In \cite{OQS},  the determinantal point processes $\mathscr{P}_{z,z'}$ are obtained as scaling limits of determinantal point processes on $\Conf(\Z')$ with hypergeometric kernel measures, where $\Z' = \frac{1}{2}+ \Z$ is the set of half-integers. Being point processes all concentrate on the set of {\it balanced configurations} with a finite number of particles:
$$
\{X \in \Conf(\Z'): \#(X\cap \Z'_{+}) = \#(X\cap \Z'_{-})< \infty\},
$$
the determinantal point processes with hypergeometric kernels are of course  {\it balanced rigid} in the sense of Theorem A.  However, as already shown in Hermitian kernel case, the rigidity property is not stable under taking limits. Indeed, orthogonal polynomial ensembles,  having a fixed number of particles, are  rigid in the sense of Ghosh \cite{Ghosh-sine} and Ghosh-Peres \cite{Ghosh-rigid}, while in general this is not the case for their scaling limits. For example, as Holroyd and Soo \cite{Soo} showed,  the  determinantal point process on the unit disk  $\D$ with Bergman kernel: 
$$
K_\mathrm{Berg}(z,w)= \frac{1}{\pi (1 -z \bar{w})^2}, \quad z, w \in \D,
$$
is not rigid (the Radon-Nikodym derivatives between this measure and its Palm measures are computed in \cite{QB3}), but  is nonetheless the limit of the following sequence of rigid determinantal point processes whose kernels are given by finite rank orthogonal projections:
$$
K_\mathrm{Berg}^{(n)}(z,w)= \frac{1}{\pi} \sum_{k = 0}^{n-1} (k + 1) (z\bar{w})^k, \quad z, w \in \D.
$$
\end{rem}

\subsection{Main results for general \texorpdfstring{$J$}{a}-Hermitian kernels}\label{intro-sec-J}
Our proofs of Theorem A and Theorem B do not proceed by limit transition from processes with finitely many particles and work for more general $J$-Hermitian kernels. 

We briefly recall the necessary definitions.
Let $P_{+}, P_{-}$ denote the orthogonal projections on $L^2(\R) = L^2(\R, dx)$ whose ranges are the subspaces $L^2(\R_{+})$ and $L^2(\R_{-})$ respectively.  Define a bounded linear operator $J$ on $L^2(\R)$ by $$J : = P_{+}- P_{-}.$$ Introduce  an indefinite $J$-scalar product $[\cdot, \cdot]$ on $L^2(\R)$ by the formula
$$
[ f, g] : = (Jf, g) = (P_{+}f, P_{+}g) - (P_{-}f, P_{-}g), \quad f, g \in L^2(\R),
$$
where $(\cdot, \cdot)$ denotes the usual scalar product in $L^2(\R)$. 
A bounded linear operator $K$ on $L^2(\R)$ is called $J$-self-adjoint if $[ Kf, g] = [ f, Kg]$ for any pair $f, g \in L^2(\R)$. By slightly abusing the notation, the  kernel of an integral operator $K$ will denote again by $K$. A kernel $K$ is called a $J$-Hermitian kernel, if the corresponding operator $K$ is $J$-self-adjoint. More precisely, $K(x,y)$ is $J$-Hermitian if it induces a bounded linear operator and if 
\begin{align}\label{J-herm}
 K(x,y) = \sgn(x) \sgn(y) \overline{K(y,x)}, \quad x, y \in \R^*,
\end{align}
where $\sgn(x)$ is the sign of the real number $x\in \R^*$.

By convention, a bounded measurable function $f: \R^* \rightarrow \C$ will be identified with the bounded linear operator $M_f$ on $L^2(\R)$ defined by 
$$
M_f (g) = fg, \text{ for any }g \in L^2(\R).
$$
The notation $fK$ (or $f \cdot K)$ and $Kf$ (or $K\cdot f$) stands for the composition operators $M_f \circ K$ and $K \circ M_f$ respectively.

Given a bounded linear operator $K$ on $L^2(\R)$, we set
\begin{align}\label{hat-o}
\widehat{K} : = \sgn \cdot  K + \chi_{\R_{-}} = P_{+} K + P_{-}(1 - K).
\end{align}
An operator $K$ is $J$-self-adjoint if and only if the operator $\widehat{K}$ is self-adjoint in the usual sense.

The following Theorem of Lytvynov gives a necessary and sufficient condition for the existence of a determinantal point process with a given $J$-Hermitian kernel.
\begin{thm}[E. Lytvynov \cite{Lytvynov-J}]\label{thm-0}
Let $K$ be a $J$-self-adjoint bounded linear operator on $L^2(\R)$. Assume that the operators $P_{+}KP_{+}$ and $P_{-}KP_{-}$ are non-negative. Assume also that, for any  bounded subsets $\Delta_1, \Delta_2$ of $\R$ such that $\Delta_1\subset \R_{+}$ and $\Delta_2 \subset \R_{-}$, the operators $\chi_{\Delta_i} K \chi_{\Delta_i} (i = 1, 2)$ are in trace-class, while $\chi_{\Delta_2}K\chi_{\Delta_1}$ is Hilbert-Schmidt. Then the integral kernel $K(x,y)$ of the operator $K$ is the correlation kernel of a determinantal point process on $\R$ if and only if $0\le \widehat{K}\le 1$.
\end{thm}

The determinantal point process induced by a correlation kernel $K$ as in Theorem \ref{thm-0} will be denoted by $\PP_K$.

\subsubsection{Theorem A for general \texorpdfstring{$J$}{a}-Hermitian kernels}
We now formulate a general variant of Theorem A, namely, a sufficient condition on the $J$-Hermitian kernel $K$ for the determinantal point process $\PP_K$ to be balanced rigid.  For the purpose of our later use of Fourier analysis, we state our result in this part only for determinantal point processes with a singularity at infinity. The case with a single singularity at the origin, such as the determinantal point processes with Whittaker kernels, can be easily transformed to this case by the change of variables $x \mapsto 1/x$. 

We need the following two conditions on the kernel $K$ : Condition \ref{con1} guarantees that the $J$-Hermitian kernel $K$ is indeed a correlation kernel of a certain determinantal point process and the variance of the linear statistics 
$$
\sum_{x \in X} \sgn(x) \varphi(x), \quad (\text{where $X \in \Conf(\R)$}),
$$
can be expressed by a simple formula, see Lemma \ref{lem-var} below. Condition \ref{con2} guarantees that the diagonal coefficient $K(x,x)$ is locally integrable on $\R$ and controls the rate of decay of off-diagonal coefficients $K(x,y)$ when $|x-y|$ is large. The former condition on $K(x,x)$ implies in particular that the associated determinantal point process has no accumulation point at any point of the real line. 

\begin{con}\label{con1}
Assume that $K$ is the integral kernel of a bounded linear operator on $L^2(\R)$ such that
\begin{itemize}
\item  the operators $P_{+}KP_{+}$ and $P_{-}KP_{-}$ are non-negative. Moreover, for any bounded subsets $\Delta_1, \Delta_2$  of the real line such that $\Delta_1\subset \R_{+} $ and $\Delta_2 \subset \R_{-}$, the operators $\chi_{\Delta_i} K \chi_{\Delta_i} (i = 1, 2)$ are in trace-class, and the operator $\chi_{\Delta_2}K\chi_{\Delta_1}$ is Hilbert-Schmidt. 
\item the following operator 
$$
\widehat{K} := \sgn \cdot K +  \chi_{\R_{-}}
$$
 defines an {\it orthogonal projection} on $L^2(\R, dx)$.  
 \end{itemize}
 \end{con}

\begin{con}\label{con2} 
 Fix $M>0$. Assume  that the kernel $K$ satisfies the following conditions: 
\begin{itemize}
\item Let $K(x,x)$ be the diagonal value of the kernel $K$, then for any $R>0$, we have
\begin{align}\label{no-singularity}
\int_{| x| \le R} K(x,x)dx<\infty;
\end{align}
\item There exists a non-negative integrable function $\Phi \in L^1(\R, dt) $ satisfying
\begin{align}\label{tail-condition}
\int_{| t | \ge R } \Phi(t) dt = O (R^{-1}) \text{\, as \,} R \to \infty,
\end{align}
such that if $| x| \ge M, | y| \ge M$, then
\begin{align}\label{pt-es}
 | K(x,y)|^2 \le \Phi(x-y);
\end{align}
\end{itemize}
\end{con}

The general variant of Theorem A is
\begin{thm}[A sufficient condition for balanced rigidity]\label{intro-thm-rigid}
Let $K$ be a $J$-Hermitian kernel satisfying Conditions \ref{con1} and \ref{con2}.  Then the determinantal point process $\PP_K$ possesses the following rigid property: for any bounded Borel subset $B \subset \R$, there exists a measurable function $N_B^{out}: \Conf(\R) \rightarrow \Z$, such that for $\PP_K$-almost every $X\in \Conf(\R)$, we have
$$
\#(B \cap X \cap \R_{+}) -  \#(B \cap X \cap \R_{-})  = N^{out}_B(X\setminus B).
$$
\end{thm}

 In particular, if $B \subset \R_{+}$ is a bounded subset in the positive semi-axis, then $\#(B \cap X)$ is $\PP_K$-almost surely determined by $X \cap (\R\setminus B)$. If $B\subset\R_{-}$ is a bounded subset in the negative semi-axis, the same result holds.

Using  a variant of Proposition 8. 1  in \cite{QB3}, we derive from Theorem \ref{intro-thm-rigid}  the following corollary.
\begin{cor}
Let $K$ be a $J$-Hermitian kernel satisfying Conditions \ref{con1} and \ref{con2}.  Let $n, k$ be two non-negative integers such that $n\ne k$. Then for Lebesgue-almost every $\mathfrak{p} = (p_1^{+}, \dots, p_n^{+}; p_1^{-}, \dots, p_k^{-}) \in \R_{+}^n \times \R_{-}^k $ of distinct points,  the reduced Palm measure 
$\PP_{K}^{\mathfrak{p}}$ and the initial determinantal measure $\PP_K$ are mutually singular.
\end{cor}

\subsubsection{Theorem B for general \texorpdfstring{$J$}{a}-Hermitian kernels}
We now formulate a general variant of Theorem B, namely, a sufficient condition for the determinantal point process to have {\it balanced Palm equivalence property} in the sense of Theorem B, see also Definition \ref{defn-Palm-equivalence} below.  In this part, let us state the result in the case where there is a single singularity at the origin (rather than  a singularity at infinity).

We first need the definition of $L$-processes of Borodin and Olshanski.

\begin{defn}[$L$-kernel]\label{defn-L-kernel}
Given a bounded linear operator $K$ on $L^2(\R)$, if $1 - K$ is invertible, then we define the $L$-operator of $K$ by 
\begin{align*}
L = K(1-K)^{-1}. 
\end{align*}
 In order to emphasize that the operator $K$ depends on $L$, we will sometimes write $K_L$ instead of $K$, thus having 
$$
K_L = L(1+L)^{-1}.
$$
\end{defn}

\begin{con}\label{con35}
Assume that $L$ is a bounded operator on $L^2(\R^*,dx)$ having the following block form: 
\begin{align}\label{L-integrable}
 L(x,y) = \left[\begin{array}{cc} 0  &    V\\  -V^*& 0   \end{array}\right]= \left[\begin{array}{cc} 0  &    \frac{  A^{+}(x) A^{-}(y)}{x-y}\\   \frac{  A^{-}(x) A^{+}(y)}{x-y} & 0   \end{array}\right], 
 \end{align}
where $A \in C^1(\R^*) \cap L^2(\R^*, dx)$ is a real-valued function and 
$$
A^{+} := A \chi_{\R_{+}} \an A^{-}: = A \chi_{\R_{-}}.
$$ 
In other words, the operator $L$ admits an integral kernel given by 
\begin{align}\label{L-kernel-int}
L(x,y) =  \frac{A^{+}(x) A^{-} (y) + A^{-} (x) A^{+}(y)}{x-y} ,\, x, y \in \R^*.
\end{align}  
 We will also assume that the support of the funcition $A$ is the whole punctured line $\R^*$.
\end{con}

\begin{lem}\label{lem-V-intro}
Let $L$ be an operator satisfying Condition \ref{con35}.  Then the operator $K_L$ is a $J$-self-adjoint operator satisfying all the conditions of  the Lytvynov's Theorem \ref{thm-0}. In particular, $K_L$ is the correlation kernel of a determinantal point process on $\R^*$.
\end{lem}

\begin{defn}[$L$-processes]
We denote by $\mu_L$ the determinantal point process on $\R^*$ whose correlation kernel is $K_L = L(1+L)^{-1}$, that is 
\begin{align}\label{muL}
\mu_L : = \PP_{K_L}.
\end{align}
Following Borodin and Olshanski, such processes will be called $L$-processes.
\end{defn}

We need the following auxiliary propositions.

Let $\mathfrak{p} =  (p^{+}_1, \dots p^{+}_n;  p^{-}_1, \dots, p^{-}_n)$ be a $2n$-tuple of real numbers  such that $p_i^{+} > 0$ and $p^{-}_i <0$ for $i = 1, \dots, n$. Moreover, assume that $p_i^{+} \ne p_j^{+}, p_i^{-} \ne p_j^{-}$ when $i \ne j$.  Define 
\begin{align}\label{fp}
f_{\mathfrak{p}} (x) =    \prod_{i = 1}^n \left(\frac{x/p_i^{+}  - 1 }{x/p_i^{-} - 1} \chi_{\R_{+}}(x)    + \frac{x/p_i^{-}  - 1 }{x/p_i^{+} - 1} \chi_{\R_{-}}(x) \right).
\end{align}

\begin{prop}\label{intro-Palm-prop}
Let $L$ be an operator satisfying Condition \ref{con35}.  Let $\mu_L$ be the determinantal point process on $\R^*$ whose correlation kernel is $K_L = L(1+L)^{-1}$. Then the reduced Palm measure $\mu_L^{\mathfrak{p}}$ conditioned at a $2n$-tuple of distinct points 
\begin{align}\label{intro-p}
\mathfrak{p} = (p_1^{+}, \dots, p_n^{+}; p_1^{-}, \dots, p_n^{-}) \in \R_{+}^n \times \R_{-}^n, 
\end{align}
is again an $L$-process and is given by 
\begin{align}\label{intro-palm-rel}
\mu_L^{\mathfrak{p}} = \mu_{f_{\mathfrak{p}} L f_{\mathfrak{p}}}. 
\end{align}
\end{prop}

\begin{prop}
Let $L$ be an operator satisfying Condition \ref{con35} and let $\mathfrak{p}$ be a $2n$-tuple of distinct points in $\R^*$ given as in \eqref{intro-p}, the function $f_{\mathfrak{p}}$ is defined by formula \eqref{fp}. Then the following limit 
\begin{align*}
\overline{S}_{\mathfrak{p}}(X) : & =  \lim_{\delta\to 0^{+}}    \bigg\{   \sum_{i = 1}^n   \Big(\sum_{x \in X \cap (\delta, \infty)}\log \left| \frac{x/p_i^{+}  - 1 }{x/p_i^{-} - 1} \right|^2 - \sum_{x \in X\cap (-\infty, - \delta) } \log \left| \frac{x/p_i^{-}  - 1 }{x/p_i^{+} - 1} \right|^2\Big)  
\\
& - \E_{ \mu_L}  \sum_{i = 1}^n   \Big(\sum_{x \in X \cap (\delta, \infty)}\log \left| \frac{x/p_i^{+}  - 1 }{x/p_i^{-} - 1} \right|^2 - \sum_{x \in X\cap (-\infty, - \delta) } \log \left| \frac{x/p_i^{-}  - 1 }{x/p_i^{+} - 1} \right|^2\Big)
\bigg\}
\end{align*}
exists for $ \mu_L$-almost every configuration $X\in \Conf(\R^*)$.  Moreover, we have
$$
 \exp ( \overline{S}_{\mathfrak{p}} ) \in L^1(\Conf(\R^*),  \mu_L).
$$
\end{prop}

\begin{thm}[A sufficient condition to have balanced Palm equivalence property]\label{intro-thm-equi}
Let $L$ be an operator satisfying Condition \ref{con35} and let $\mathfrak{p}$ be a $2n$-tuple of distinct points in $\R^*$ given as in \eqref{intro-p}, let  $f_{\mathfrak{p}}$ denote the function defined by formula \eqref{fp}. 
Then  the reduced Palm measure $\mu_{L}^{\mathfrak{p}}$  is equivalent to $\mu_L$. For the Radon-Nikodym derivative, we have the $\mu_L$-almost sure equality:
$$
\frac{d \mu_L^{\mathfrak{p}}}{d  \mu_L } (X)=  \frac{\exp ( \overline{S}_{\mathfrak{p}}(X) )
}{\E_{\mu_L} \Big[\exp ( \overline{S}_{\mathfrak{p}} )
 \Big]}.
$$
\end{thm}

\subsection{Olshanski's Problem}\label{sec-Olshanski-program}

Olshanski \cite{OQS} posed the following

\begin{problem}
Let $P_1$ and $P_2$ be two determinantal point processes on a common phase space with correlation kernels $K_1(x,y)$ and $K_2(x,y)$ respectively.  Decide the equivalence and the mutual singularity relations between $P_1$ and $P_2$ by inspection of their correlation kernels $K_1(x,y)$ and $K_2(x,y)$.  When $P_1$ and $P_2$ are equivalent, calculate the Radon-Nikodym derivative between them.
\end{problem}

We now briefly mention the previous works on this problem for projection kernels and note the particle-hole duality relation, in the case of discrete phase spaces, of these results to the results of the present paper.

\begin{itemize}
\item {\bf The Gamma-kernel.}

 Olshanski \cite{OQS} obtained the quasi-invariance of the so-called Gamma kernel determinantal point processes on  the space $\Z'$ of half-integers under the action of the group $S(\infty)$ of  finite permutations of arbitrary size.  The group  $S(\infty)$ acts naturally on $\Z'$ and hence on the space of configurations over $\Z'$.  Let  $P_1$ be the Gamma kernel determinantal point process on $\Z'$ with the correlation kernel denoted by $K_1$, see \cite{BO-R} for the precise definition. Take an element $\sigma\in S(\infty)$, denote $P_2 = \sigma_{*}(P_1)$ the determinantal point process obtained by the transformation $\sigma$ on the space of configurations $\Conf(\Z')= 2^{\Z'}$ . Then $P_2$ has a correlation kernel given by 
$$
K_2(x,y) = K_1(\sigma^{-1}(x), \sigma^{-1}(y)) \quad x, y \in \Z'.
$$ 
By {\it limit transition} from finite particle systems, Olshanski proved the equivalence of $P_1$ and $P_2$ and calculated the Radon-Nikodym derivative between them as a multiplicative functional.   

\item {\bf Determinantal point processes  with integrable kernels.} 

In \cite{BQS}, it has been proved that all determinantal point processes on $\R$ (or $\Z$) whose correlation kernels $K$ are Hermitian orthogonal projections and have an integrable form as follows:
\begin{align}\label{integrable-form}
K(x,y)  = \frac{A(x) B(y) - A(y) B(x)}{x-y}, \quad x, y \in \R (\text{\,or\,} x, y \in \Z),
\end{align}
are quasi-invariant, under the action of the group of compactly supported diffeomorphisms on $\R$ (or under the action of $S(\infty)$ on $\Z$).  The equivalence between {\it reduced Palm measures of the same order} plays a central r\^ole in the proof, which proceeds by the method, further developed in this paper, of {\it regularized multiplicative functionals}.

\item {\bf Determinantal point processes associated with Hilbert spaces of holomorphic functions.}

Holroyd and Soo \cite{Soo} have shown that the determinantal point process with the Bergman kernel on the unit disk has the property of {\it insertion tolerance}: its Palm measures are equivalent to itself. 
For the Ginibre point process on the complex plane, using its finite-dimensional approximations by orthogonal polynomial ensembles, Osada and Shirai \cite{OS14} have shown that Palm measures of different orders are singular, while  Palm measures of the same orders are equivalent and  the Radon-Nikodym derivative is a regularized multiplicative functional.
In \cite{QB3}, the method of  {\it regularized multiplicative functionals} has been further elaborated for obtaining in a unified way, on one hand, the equivalence of reduced Palm measures {\it of the same order} of the determinantal point processes on the complex plane $\C$ with correlation kernels given by the reproducing kernels of  generalized Fock spaces on $\C$, and on the other hand, the equivalence of reduced Palm measures {\it of all orders} of the determinantal point processes on the open unit disk $\D\subset \C$ with correlation kernels given by  the reproducing kernels of generalized Bergman spaces on $\D$. Specifically, the Radon-Nikodym derivative between the determinantal point process with the Bergman kernel on the unit disk and its Palm measures is computed explicitly as a regularized Blaschke product.

As a consequence, we also obtained the quasi-invariance property of these determinantal point processes, under the action of the group of compactly supported diffeomorphisms on the complex plane $\C$ and on the open unit disk $\D$ respectively.  

\item {\bf Relations with rigidity of determinantal point processes.} 

 Recall that a  point process on a Euclidean space $\R^d$ is said to be rigid in the sense of Ghosh \cite{Ghosh-sine} and Ghosh-Peres\cite{Ghosh-rigid}, if for any bounded open subset $ B \subset \R^d$, such that the topological boundary $\partial B$ is negligible with respect to the $d$-dimensional Lebesgue measure, the number of particles of this point process inside the subset $B$ is almost surely determined  by the configuration outside of the subset $B$.     Olshanski's problem is closely related to this rigidity property of determinantal point processes. In particular,  if a determinantal point process is rigid in the above sense, then its reduced Palm measures {\it of different orders} are almost surely singular,  see \cite[Prop. 8.1]{QB3}.  Note that for processes with $J$-Hermitian kernels we encounter a rather different notion of rigidity. In the case of discrete phase spaces, however, the new notion can be reduced to the old one, as we shall now demonstrate.

\item {\bf Discrete phase spaces and the particle-hole duality}

Analogues of our main results, Theorem \ref{intro-thm-rigid} and  Theorem \ref{intro-thm-equi}, can be formulated and proved in similar way when the phase space $\R^*$ is replaced by $\Z$ or $\Z' = \frac{1}{2}+ \Z$ or any other discrete subsets in $\R$.  In particular, in the case where the phase space is a discrete subsets of $\R$, our results are related to previous works \cite{OQS, BQS}  by doing {\it particle-hole duality} as follows: Let $\PP$ be a determinantal point process, say on $ \Z'$. Assume that $\PP$ has a Hermitian correlation kernel  $K$. Define the particle-hole duality on $\Z'_{-} = \Z' \cap \R_{-}$ as a map $dual: \Conf(\Z') \rightarrow \Conf(\Z')$ given by 
$$
dual(X) :=  (X  \cap \Z_{+}') \sqcup  (\Z_{-}'\setminus X).
$$
Then this particle-hole duality transform $\PP$ to  a new point process $ dual_{*}(\PP)$, which is again a determinantal point process. A correlation kernel of this new  point process can be provided by 
$$
K^{\circ} : = \sgn \cdot K +  \chi_{\Z'_{-}}.
$$
Note that $K^\circ$ is $J$-Hermitian with respect to the partition $\Z' = \Z'_{+} \sqcup \Z'_{-}$ and the orthogonal decomposition $\ell^2(\Z') = \ell^2(\Z'_{+}) \oplus \ell^2(\Z'_{-})$. In general, the particle-hole duality transforms a rigid  point process (see Definition \ref{rigid-GP}) to a {\it balanced rigid} one and vice-versa. It transforms a quasi-invariant point process to a point process having balanced Palm equivalence property and vice-versa. In terms of correlation kernels, the particle-hole duality transforms Hermitian kernels to $J$-Hermitian ones and vice-versa. 

At the same time, we would like to note that  the particle-hole duality argument only works in the case where the phase spaces are discrete.  This can be already seen on the level of correlation kernels, indeed, the kernel $\widehat{K}$ defined in \eqref{hat-o} corresponds to $K^\circ$ as above. Observe that $\widehat{K}$ can not be used to define (extended) Fredholm determinants,  and it  is not the correlation kernel of any determinantal point process. Thus when the phase space is $\R^*$,  processes with $J$-Hermitian kernels can not be transformed to processes with Hermitian kernels. 

\end{itemize}


\subsection{Organization of the paper and schemes of proofs}
The paper is organized as follows. 

In the preliminary part, \S \ref{section-pre}, we briefly recall the definition of determinantal point processes and theory of reduced Palm measures. In particular, we  collect the necessary results from \cite{Lytvynov-J} on the general determinantal point processes with $J$-Hermitian correlation kernels.  Some standard properties of extended Fredholm determinants are also collected in \S \ref{section-pre}. The proofs for these properties are postponed to the appendix in the end of the paper. 

 The main body of the paper is separated into two parts.  The first part,  \S \ref{Balanced-rigidity}, is devoted to the proofs of Theorem \ref{intro-thm-rigid} and Theorem A; the second part,  \S \ref{tolerance-sec}, is devoted to the proofs of Theorem \ref{intro-thm-equi} and Theorem B. These two parts are essentially independent from each other.  

$\emph{First part \S \ref{Balanced-rigidity}: proofs of Theorem \ref{intro-thm-rigid} and Theorem A.}$ Here we follow the scheme of Ghosh and Peres \cite{Ghosh-rigid}.  Let $\PP_K$ be the determinantal point process as in Theorem \ref{intro-thm-rigid}.  Our main task is to construct, after fixing an arbitrarily large interval $U_R = [-R, R]$,  a sequence of {\it compactly supported} continuous functions $(\varphi_n)_{n\ge 1}$ defined on $\R$, such that $\varphi_n(x)$ tends to $1$ uniformly on $U_R$ when $n$ tends to infinity. Moreover, the following limit relation holds:
$$
\Var_{\PP_K}\left(\sum_{x \in  X} \sgn(x) \varphi_n(x)\right) \xrightarrow{n\to \infty} 0.
$$
See also \cite{QB4, QB3, Buf-rigid} for the use of the same method  in the Hermitian case.

$\emph{Second part \S \ref{tolerance-sec}: proofs of Theorem \ref{intro-thm-equi} and Theorem B.}$ 
There are  three main ingredients in the proofs of our main results in this part:
\begin{itemize}
\item[(i)] The $J$-Hermitian kernels for the determinantal point processes under consideration have $L$-kernels, that is, the determinantal point processes are $L$-processes; 
\item[(ii)] Under certain assumptions on the $L$-kernel of the initial determinantal point process $\mu_L$, all the reduced Palm measures conditionned at an equal number of positions at both sides of $\R^*= \R_{+}\sqcup \R_{-}$ are again $L$-processes, and  the $L$-kernels for these reduced Palm measures have the form $f L f$, where $f$ is certain bounded measurable function defined on $\R^*$.
\item[(iii)] Under suitable assumptions on the kernel $L$ and $f$, the two determinantal point processes $\mu_{fLf}$ and $\mu_L$ are proved to be equivalent and the Radon-Nikodym derivative $d\mu_{fLf} /d\mu_L$ can be  computed explicitly as a {\it regularized multiplicative functional}.
\end{itemize}

The verification of part (ii) will be given in \S \ref{sec-L-rel}.  The proof relies heavily on the algebraic structures of the $L$-kernels, see Condition \ref{con-four} in \S \ref{sec-L-rel}. 

Let us now explain part (iii). We will first prove a preliminary and abstract version in Proposition \ref{prop-pre}:  under a certain condition on the $L$-kernel,  if $\supp(f-1) \subset \R^*$ has a positive distance from the origin , then $\mu_{fLf}$ and $\mu_L$ are equivalent and the Radon-Nikodym derivative is given by a normalized multiplicative functional: 
\begin{align}\label{simple-rn}
\frac{d\mu_{fLf}}{d\mu_L}(X) =  \frac{\prod\limits_{x \in X}|f(x)|^2}{ \E_{\mu_L}  \prod\limits_{x \in X}|f(x)|^2 }.
\end{align}
But for $f_{\mathfrak{p}}$ defined by \eqref{fp},  $\supp(f_{\mathfrak{p}}-1) = \R^*$, it does not have positive distance from the origin. Moreover, the usual multiplicative functional $\prod\limits_{x \in X}|f_{\mathfrak{p}}(x)|^2$ on the right-hand side of \eqref{simple-rn} does not converge absolutely.  For overcoming this difficulty, we are led to use a new version of {\it regularized multiplicative functionals}. One ingredient in the formalism of this new version of regularized multiplicative functionals is the use of the twisted linear statistics:
$$
\sum_{x\in X} \sgn(x) \varphi(x),  \quad (\text{where $X \in \Conf(\R^*)$}).
$$
Extra efforts are also required in dealing with the extended version of Fredholm determinants. 
 The reader is referred to \cite{BQS, QB3} for the use of another version of regularized multiplicative functionals in computing Radon-Nikodym derivatives between determinantal point processes whose correlation kernels are Hermitian. 


\section{Preliminaries}\label{section-pre}

Let $\mathscr{E}$ be a locally compact complete metrizable separable space. Assume that on $\mathscr{E}$ is equipped with a positive $\sigma$-finite Borel measure $\mu$.  A {\it configuration} on $\mathscr{E}$ is defined to be an $\N\cup\{0\}$-valued Radon measure on $\mathscr{E}$; in other words, a configuration on $\mathscr{E}$ is a  collection of {\it particles}, possibly with multiplicity,  that  admits no accumulation points in $\mathscr{E}$.   Let $\Conf(\mathscr{E})$ denote the space of all  configurations on $\mathscr{E}$. With respect to the topology induced by the vague topology on the space of Radon measures on $\mathscr{E}$, the space $\Conf(\mathscr{E})$ is itself a complete metrizable separable space.
A {\it point process} on $\mathscr{E}$ is by definition a Borel probability measure on $\Conf(\mathscr{E})$. For further background on point processes, see, e.g., \cite{Daley-Vere}.

We now briefly recall the definition of {\it determinantal point processes}, see, e.g., \cite{DPP-M, DPP-S, DPP-L}. Fix a Radon measure $\mu$ on $\mathscr{E}$. A determinantal point process on $\mathscr{E}$ is  determined by a correlation kernel $K$, that is, a certain two-variable complex-valued function $K(x,y)$ on $\mathscr{E} \times \mathscr{E}$. More precisely, if we denote the determinantal point process with a correlation kernel $K$ by $\PP_K$, then this measure $\PP_K$ is completely determined by the following: for any positive integer $k \ge 1$ and any disjoint bounded subsets $D_1, 
\cdots, D_k$ of $\mathscr{E}$, we have 
\begin{align}\label{def-DPP}
 \!\!\! \int\limits_{\Conf(\mathscr{E})} \!\!\!  \prod_{i =1}^k \#(X \cap D_i) \PP_K(d X)= \!\! \!\!\!\!  \int\limits_{ D_1 
\times \cdots \times D_k}  \!\!\!  \!\! \!\! \det(K(x_i, x_j))_{1 \le i, j \le k}  d\mu^{\otimes k}(x_1 \cdots x_k).
\end{align}
The equality \eqref{def-DPP} implies that  if $D_1, \cdots, D_r$ are disjoint bounded subsets of $\mathscr{E}$ and $k_i$ are positive integers, $k = k_1 + \cdots  + k_r$, then
\begin{align*}
& \int\limits_{\Conf(\mathscr{E})} \prod_{i =1}^r \#(X \cap D_i)  (\#(X \cap D_i) -1) \cdots (\#(X \cap D_i) - k_i +1) \PP_K(d X)
\\
 & = \int\limits_{ D_1^{k_1} \times \cdots \times D_r^{k_r}}    \det(K(x_i, x_j))_{1 \le i, j \le k}  d\mu(x_1) \cdots d\mu(x_k).
\end{align*}
See, e.g., \cite[Remark 1.2.3 ]{HKPV}.   By definition,  determinantal measures  are always supported on the subset of  {\it simple} configurations, that is, configurations all of whose particles have multiplicity one.

In this paper, we are particularly interested in the determinantal point processes with $J$-Hermitian correlation kernels, see \S \ref{intro-sec-J} and \S \ref{section-J} for a brief introduction on $J$-Hermitian kernels. The reader is referred to \cite{Lytvynov-J} for the general theory of such point processes.

\subsection{Palm measures of determinantal point processes}\label{sec-Palm}
In what follows, by Palm measures of a point process, we always mean its {\it reduced} Palm measures. Let us briefly recall the definition of Palm measures of determinantal point processes.  For further details on Palm measures of general point processes,  the reader is referred to \cite{Kallenberg, Daley-Vere}

Let $\PP$ be a point process on $\mathscr{E}$ (later, we will focus on the case $\mathscr{E}$= $\R$ or $\R^*$). Assume that for any positive integer $k$, the point process $\PP$ admits the $k$-th correlation measure $\rho_k$ on $\mathscr{E}^k$, that is, $\rho_k$ is a positive measure on $\mathscr{E}^k$ such that for any disjoint bounded subsets $D_1, \cdots, D_k$ of $\mathscr{E}$, the following identity 
\begin{align}\label{cor-measure}
\int\limits_{\Conf(\mathscr{E})} \prod_{i =1}^k \#(X \cap D_i) \PP(d X)=   \int\limits_{D_1 
\times \cdots \times D_k}  \rho_k(dx_1 \cdots dx_k)
\end{align}
holds.  Then for $\rho_k$-almost every $k$-tuple $\mathfrak{q} = (q_1, \dots, q_k) \in \mathscr{E}^k$ of {\it distinct} points in $\mathscr{E}$, one can define a point process   $\PP^{\mathfrak{q}}$ on $\mathscr{E}$ by the following disintegration formula: for any non-negative Borel test function $u: \Conf(\mathscr{E}) \times \mathscr{E}^k\rightarrow \R $, 
\begin{align}\label{def-Palm}
\int\limits_{\Conf(\mathscr{E})}  \sum_{q_1, \dots, q_k \in X}^{*} u(X; \mathfrak{q}) \PP(d X)  =    \int\limits_{\mathscr{E}^k} \rho_k(d\mathfrak{q}) \!\int\limits_{\Conf(\mathscr{E})} \!  u (X \cup \{q_1, \dots, q_k\};  \mathfrak{q})  \PP^{\mathfrak{q}}(dX),
\end{align}
where $\sum\limits^{*}$ denotes the sum over all distinct points $q_1, \dots, q_k \in X$. The point process $\PP^{\mathfrak{q}}$ is called the Palm measure of $\PP$ conditioned at $q_1, \dots, q_k$.

In the above situation,  if the $k$-th correlation measure $\rho_k$ for the point process $\PP$ is absolutely continuous with respect to the product measure $\mu^{\otimes k}$ on $\mathscr{E}^k$, then  the Radon-Nikodym derivative 
$$
f_k(x_1, \cdots, x_k): = \frac{d \rho_k}{d\mu^{\otimes k}}(x_1, \cdots, x_k)
$$
is called the $k$-th correlation function of $\PP$.  In terms of correlation functions, the Palm measure $\PP^{\mathfrak{q}}$ can be described as follows: it is a point process on $\mathscr{E}$ such that its $n$-th correlation function is given by 
 $$
 f_n^{\mathfrak{q}} (x_1, \cdots, x_n)= f_{n+k} (q_1, \cdots, q_k, x_1, \cdots, x_n).
 $$

Informally, if $\mathcal{X}$ is a random configuration on $\mathscr{E}$ whose probability distribution is given by the point process $\PP$, then $\PP^{\mathfrak{q}}$ is the conditional distribution of the random configuration $\mathcal{X} \setminus \{q_1, \dots, q_k\}$ conditioned to the event that all particles $q_1, \dots, q_k $ are in the configuration $\mathcal{X}$.

A Theorem of Shirai and Takahashi  \cite{ST-palm} states that the Palm measures of a determinantal measure are again determinantal measures. Let us formulate this result more precisely. Assume now $\PP$ is  a determinantal point process on $\mathscr{E}$ induced by a correlation kernel $K$, that is, $\PP= \PP_K$.  Let $q \in \mathscr{E}$ and assume that $K(q,q)>0$. Set  
\begin{align}\label{1-rank}
K^q(x,y) = K(x,y) - \frac{K(x,q) K(q,y)}{K(q,q)}.
\end{align}
If  $K(q,q)= 0$, we set $K^q = K$. More generally, if $\mathfrak{q} = (q_1, \dots, q_k) \in \mathscr{E}^k$ is a $k$-tuple of  distinct points in $\mathscr{E}$, then we define by iteration
\begin{align}\label{Kq-iteration}
K^{\mathfrak{q}} = (\cdots (K^{q_1})^{q_2}\cdots)^{q_k}.
\end{align}
Observe that the order of the points $q_1, q_2, \cdots q_k$ has no effect in the above iteration.

\begin{thm}[Shirai and Takahashi \cite{ST-palm}]\label{ST}
Let $\PP= \PP_K$ be a determinantal point process on $\mathscr{E}$ induced by a correlation kernel $K$.  Let $k \in \N$ be a positive integer.  Then for $\rho_k$-almost every $k$-tuple $\mathfrak{q}  \in \mathscr{E}^k$ of distinct points in $\mathscr{E}$, the Palm measure $\PP_K^{\mathfrak{q}}$ of $\PP_K$ conditioned at $\mathfrak{q}$ is again a determinantal point process on $\mathscr{E}$. Moreover, $\PP_K^{\mathfrak{q}}$ is induced by the kernel $K^{\mathfrak{q}}$ defined in \eqref{Kq-iteration}, that is, we have
$$
\PP_K^{\mathfrak{q}} = \PP_{K^{\mathfrak{q}}}.
$$
\end{thm}

\begin{rem}
Theorem \ref{ST} was proved by Shirai and Takahashi in \cite{ST-palm} for determinantal point processes with Hermitian correlation kernels. This result was independently proved by Lyons in \cite{DPP-L} in the case where the phase space is a discrete countable set and the correlation kernel corresponds to a Hermitian orthogonal projection. The proof in \cite{ST-palm} can be generalized word by word for determinantal point processes without requiring that the correlation kernels are Hermitian.
\end{rem}

\subsection{\texorpdfstring{$J$}{a}-Hermitian kernels and extended Fredholm determinants}\label{section-J}
Recall that in \S \ref{intro-sec-J},  we have defined the $J$-Hermitian kernels on $\R^* = \R_{+}\sqcup \R_{-}$ as follows: a kernel $K: \R^* \times \R^*\rightarrow \C$ is called a $J$-Hermitian kernel if it defines a bounded linear operator on $L^2(\R)$ and if
\begin{align*}
 K(x,y) = \sgn(x) \sgn(y) \overline{K(y,x)}, \quad x, y \in \R^*.
\end{align*}

In Theorem \ref{thm-0}, we  recalled  the Lytvynov's characterization of the correlation kernels of determinantal point process in $J$-Hermitian case in our particular situation with the  phase space $\R^* = \R_{+}\sqcup \R_{-}$.  We shall need a slight reformulation of Theorem \ref{thm-0}.

\begin{rem}\label{two-kind-s}
Note that the determinantal point process $\PP_K$ induced by the kernel $K$ as in Theorem \ref{thm-0} accumulates at infinity (both $+\infty$ and $-\infty$), in this situation, we will say that the {\it singularity of the kernel} $K$ is at infinity.  The change of variables $x \mapsto 1/x$ on $\R^*$ transforms $\PP_K$ to a new determinantal point process on $\R^*$ induced by the new kernel $\frac{1}{|xy|} K(1/x, 1/y)$. This new determinantal point process has a single accumulation point at the origin $0 \in \R$ of the real line, and in this situation, we call that the above new kernel has a singularity at the origin. Now it is clear how to formulate a version of Theorem \ref{thm-0} when the kernel $K$ has a singularity at the origin (and there is no singularity at infinity):  we just need to replace the conditions on $\Delta_1, \Delta_2$ required in Theorem \ref{thm-0}  by  the following new condition:
$$
\text{ $\Delta_1$ and $\Delta_2$ are two measurable subsets of $\R$  both having positive distances from $0$.}
$$
Note that in the case of singularity at origin, the two subsets $\Delta_1, \Delta_2\subset \R^*$ can be unbounded.
\end{rem}

Let $\mathscr{L}_1(L^2(\R))$ denote the space of trace-class operators on $L^2(\R)$ and let $\mathscr{L}_2(L^2(\R))$ denote the space of Hilbert-Schmidt operators on $L^2(\R)$.  For further details on trace-class and Hilbert-Schmidt operators, the reader is referred to \cite{Simon-trace}. Following \cite{Lytvynov-J}, we denote by $\mathscr{L}_{1|2} (L^2(\R))$ the space of all bounded linear operators on $L^2(\R) = L^2(\R_{+}) \oplus L^2(\R_{-})$ such that when written in the following block forms 
$$
\left[ \begin{array}{cc}  a & b \\ c & d \end{array}\right],
$$
we have $a, d \in \mathscr{L}_1(L^2(\R))$ and $b,  c \in \mathscr{L}_2(L^2(\R))$. Clearly, 
\begin{align}\label{inclusions}
\mathscr{L}_1(L^2(\R)) \subset \mathscr{L}_{1|2}(L^2(\R)) \subset \mathscr{L}_2(L^2(\R)).
\end{align}

Let  $K$ be a bounded linear operator on  $L^2(\R)$. Then for any subset $\Delta \subset \R$, we denote $K^\Delta$ the compressed operator defined by
$$
 K^\Delta:  = \chi_\Delta K \chi_\Delta.
$$
 By  \cite[Prop. 12]{Lytvynov-J}, if $K$ satisfies all the conditions in Theorem \ref{thm-0}, including the condition  that $0 \le \widehat{K}\le 1$,  then for any bounded subset $\Delta \subset \R$, we have 
\begin{align}\label{K-Delta}
K^\Delta \in  \mathscr{L}_{1|2} (L^2(\R)).
\end{align}
Similarly, if $K$ satisfies all the conditions of the origin-singularity version of Theorem \ref{thm-0} as explained in Remark \ref{two-kind-s}, then for any measurable subset $\Delta\subset \R$ having a positive distance from the origin, the compressed operator $K^\Delta$ belongs to $ \mathscr{L}_{1|2} (L^2(\R))$.

The space $\mathscr{L}_{1|2}(L^2(\R))$ is a Banach space equipped with a norm $\|\cdot \|_{\mathscr{L}_{1|2}}$ defined by the following formula
\begin{align*}
\left\| \left[ \begin{array}{cc}  a & b \\ c & d \end{array}\right] \right\|_{\mathscr{L}_{1|2}}: = \| a\|_1 + \| d\|_1 + \|b\|_2 + \| c\|_2,
\end{align*}
where $\|\cdot\|_1$ is the trace-class norm while $\|\cdot\|_2$ is the Hilbert-Schmidt norm. Observe that $\mathscr{L}_{1|2}(L^2(\R))$ is not an ideal in the $C^*$-algebra $\mathscr{L}(L^2(\R))$ of all bounded  linear  operators on $L^2(\R)$. 

We collect a few standard facts needed in what follows; for the reader's convenience, we include their proofs in the Appendix.

\begin{prop}\label{prop-M}
Let   $A, B$ be two operators in $\mathscr{L}_{1|2}(L^2(\R))$. We have
\begin{align*}
\| AB\|_{\mathscr{L}_{1|2}} \le 2 \| A \|_{\mathscr{L}_{1|2}} \| B\|_{\mathscr{L}_{1|2}}.
\end{align*}
More generally, if $A_1, \cdots, A_{n}$ are operators in $\mathscr{L}_{1|2}(L^2(\R))$, then 
\begin{align*}
\| A_1 \cdot \cdots \cdot A_{n}\|_{\mathscr{L}_{1|2}} \le 2^{n-1} \prod_{i=1}^{n} \| A_i \|_{\mathscr{L}_{1|2}}.
\end{align*}
\end{prop}

\begin{prop}\label{prop-bdd-M}
Let $f: \R \rightarrow \C$ be a bounded measurable function and let $K$ be an operator in $\mathscr{L}_{1|2}(L^2(\R))$. Then the operators $fK$ and $Kf$ are both in $\mathscr{L}_{1|2}(L^2(\R))$. Moreover, 
\begin{align*}
\max (\| fK\|_{\mathscr{L}_{1|2}}, \| K f \|_{\mathscr{L}_{1|2}}  ) \le \| f\|_\infty  \| K \|_{\mathscr{L}_{1|2}},
\end{align*}
where $\| f\|_\infty$ means the $L^\infty$-norm of  $f$.
\end{prop}

\begin{prop}\label{prop-ideal}
Let   $A, B$ be two operators in $\mathscr{L}_{1|2}(L^2(\R))$. Assume that $1 + A$ is  invertible. Then the operators 
$(1 + A)^{-1} B$ and $B(1+ A)^{-1}$ both belong to the class $\mathscr{L}_{1|2}(L^2(\R))$.
\end{prop}

Recall that the trace of an operator $A\in  \mathscr{L}_1(L^2(\R))$ is given by 
$$
\tr(A) = \sum_{n = 1}^\infty (Ae_n, e_n), 
$$
where $\{e_n\}_{n=1}^\infty$ is an orthonormal basis of $L^2(\R)$. 

Let $m$ be a positive integer. Denote by $\wedge^m(L^2(\R))$ the $m$-th antisymmetric tensor power of the Hilbert space $L^2(\R)$. For any $A\in \mathscr{L}(L^2(\R))$, denote by  $\wedge^m(A)$ the unique bounded linear operator on $\wedge^m(L^2(\R))$ determined by 
$$
\wedge^m(A) (v_1 \wedge \cdots \wedge v_m) = (Av_1) \wedge \cdots \wedge (Av_m), \quad v_1, \cdots, v_m \in L^2(\R). 
$$

\begin{defn}[Fredholm determinant, Grothendieck \cite{Grothendieck-Fredholm}]
Let $A \in \mathscr{L}_1(L^2(\R))$, then the Fredholm determinant $\det(1 + A)$ is defined by 
$$
\det(1 + A): = \sum_{m=0}^\infty \tr (\wedge^m(A)).
$$
\end{defn}

In \cite{BOO-jams}, it is proven that the function $A \mapsto \det(1 + A)$ admits a unique extension to $\mathscr{L}_{1|2}(L^2(\R))$ which is continuous in the topology of $\mathscr{L}_{1|2}(L^2(\R))$. We will use the same notation $\det(1 + A)$ for this extended Fredholm determinant when $A \in \mathscr{L}_{1|2}(L^2(\R))$.

\begin{prop}\label{prop-Fm}
Let   $A, B$ be two operators in $\mathscr{L}_{1|2}(L^2(\R))$, then 
\begin{align}
\det((1 + A)(1+B)) = \det(1+A) \det(1 + B).
\end{align}
\end{prop}

\begin{prop}\label{prop-Fc}
Let $A \in\mathscr{L}_{1|2}(L^2(\R)) $ and  let $f: \R \rightarrow \C$ be a bounded measurable function. Then 
\begin{align}\label{identity-Fc}
\det( 1 + fA) = \det(1 + Af).
\end{align}
\end{prop}

We also need the following characterization of determinantal point processes with $J$-Hermitian correlation kernels in terms of multiplicative functionals.

\begin{thm}[E. Lytvynov\cite{Lytvynov-J}]\label{thm-01}
Let $K$ be a kernel as in Theorem \ref{thm-0}. Then the determinantal point process $\PP_K$ is uniquely determined by the following property: for any compactly supported bounded measurable function $f: \R \rightarrow \R$, if $\Delta \subset \R$ is a bounded subset such that $\supp(f) \subset \Delta$, then we have
\begin{align*}
\int\limits_{\Conf(\R)} \prod_{x \in X} (1 + f(x))  \PP_K(dX) =  \det(1 +       f K^\Delta).
\end{align*}  
\end{thm}


\section{Balanced rigidity}\label{Balanced-rigidity}

For any bounded Borel subset $B \subset \mathscr{E}$, let $\#_B: \Conf(\mathscr{E}) \rightarrow \N$ be defined by 
$$
\#_B(X) := \#( B \cap X).
$$
 Fix a Borel subset $C \subset \mathscr{E}$, let 
 $$
 \mathcal{F}_C  = \sigma(\{\#_B:  B \subset C, B \text{ Borel}\})
 $$
 be the smallest $\sigma$-algebra making all functions from  $\{\#_B:  B \subset C, B \text{ Borel}\}$ measurable. If $\PP$ is a point process on $\mathscr{E}$, then we denote $\mathcal{F}_C^{\PP}$ for the $\PP$-completion of $\mathcal{F}_C$.

\begin{defn}[Ghosh \cite{Ghosh-sine}, Ghosh-Peres\cite{Ghosh-rigid}]\label{rigid-GP}
A point process $\PP$ on $\R$  is called {\it rigid} if for any bounded measurable subset $B\subset \R$,   the random variable 
$\#_{B\cap \R} $
is $\mathcal{F}_{\R \setminus B}^\PP$-measurable. 
\end{defn}

\begin{defn}[Singularity at infinity version]\label{rigid-defn}
A point process $\PP$ on $\R^*$  is called {\it balanced rigid} with respect to the partition $\R^*= \R_{+} \sqcup \R_{-}$ if for any bounded measurable subset $B\subset \R^*$, the random variable 
\begin{align*}
\#_{B\cap \R_{+}} -  \#_{B\cap \R_{-}}
\end{align*}
is $\mathcal{F}_{\R^*\setminus B}^\PP$-measurable. 
\end{defn}

\subsection{A sufficient condition for balanced rigidity}

This section is devoted to the proof of Theorem \ref{intro-thm-rigid}. 

Assume that $K$ is a $J$-Hermitian kernel on $\R$ satisfying Conditions \ref{con1} and \ref{con2}.  The operators $K$ and $\widehat{K}$ have the following  block forms with respect to the decomposition $L^2(\R) = L^2(\R_{+}, dx) \oplus L^2(\R_{-}, dx)$: 
\begin{align}\label{block-KK}
K = \left[ \begin{array}{cc} K_{++} & K_{+-} \\ K_{-+} & K_{--} \end{array}\right] \an \widehat{K} = \left[ \begin{array}{cc} K_{++} & K_{+-} \\ - K_{-+} & 1_{\R_{-}} - K_{--} \end{array}\right],
\end{align}
where for instance $K_{+-}: L^2(\R_{-}, dx) \rightarrow L^2(\R^{+}, dx)$ stands for the operator $K_{+-} = P_{+}K P_{-}$ and  $1_{\R_{-}}$  stands for the identity operator on $L^2(\R_{-}, dx)$. Note that the operator $K_{+-}$ admits the following integral kernel 
$$
K_{+-}(x,y) = \chi_{\R_{+}} (x) K(x,y) \chi_{\R_{-}}(y).
$$

Recall  that in Condition \ref{con1}, we assume that the operator
\begin{align*}
\widehat{K} : = \sgn \cdot  K + \chi_{\R_{-}} = P_{+} K + P_{-}(1 - K)
\end{align*}
defines an orthogonal projection on $L^2(\R)$.

\begin{lem}\label{lem-repro}
For any $x\in \R^*$, we have  
\begin{align}\label{reproducing}
K(x,x) = \int_{\R} |K(x,y)|^2 dy.
\end{align}
\end{lem}

\begin{proof}
By assumption, $\widehat{K}$ is an orthogonal projection, hence $\widehat{K}^2 = \widehat{K}$. By substituting \eqref{block-KK} into this identity and considering the diagonal blocks, we deduce that 
$$
 \left\{ \begin{array}{c} K_{++} =  K_{++}^2 - K_{+-}K_{-+} \vspace{2mm} \\ K_{--}  = K_{--}^2 - K_{-+} K_{+-}  \end{array}. \right.
$$
The above first identity combined with \eqref{J-herm} implies  \eqref{reproducing} for $x>0$ while  the second one combined with \eqref{J-herm} implies \eqref{reproducing} when $x<0$.
\end{proof}

Given a Borel function $\varphi: \R \rightarrow \R$,  we define $\varphi^\circ: \R \rightarrow \R$ by 
\begin{align}\label{def-circ}
\varphi^\circ(x) := \sgn(x) \varphi(x).
\end{align}
By definition, the linear statistic $S[\varphi]$ corresponding to $\varphi$ is the following function on $\Conf(\R)$: 
\begin{align}\label{defn-linear-statistics}
S[\varphi] (X): = \sum_{x\in X} \varphi(x),
\end{align}
provided the right-hand side converges absolutely. For simplifying the notation, we set
\begin{align}\label{defn-sigma}
T[\varphi]: = S[\varphi^\circ].
\end{align}

Recall that by Theorem \ref{thm-0}, the kernel $K$ satisfying Condition \ref{con1} induce a determinantal point process on $\R$, denoted by $\PP_K$.  
\begin{lem}\label{lem-var}
Let $f: \R \rightarrow \R$ be a Borel function such that 
$$
\int_{\R} f(x)^2 K(x,x)dx < \infty.
$$ 
Then we have 
\begin{align}\label{variance}
\Var_{\PP_K} (T[f]) =  \frac{1}{2} \iint_{\R^2} | f(x) - f(y) |^2 | K(x,y)|^2 dxdy,
\end{align}
where $\Var_{\PP_K} (T[f])$ stands for the variance of the random variable $T[f]$ defined on the probability space $(\Conf(\R), \PP_K)$ equipped with the Borel $\sigma$-algebra.
\end{lem}

\begin{proof}
By  definition of correlation functions of determinantal point process, we have
\begin{align*}
\Var_{\PP_K} (T[f]) & = \int_{\R} f^\circ(x)^2 K(x,x)dx - \iint_{\R^2} f^\circ(x) f^\circ(y) K(x,y) K(y,x) dxdy \\ &   = \int_{\R} f(x)^2 K(x,x)dx - \iint_{\R^2} f(x) f(y) |K(x,y)|^2dxdy. 
\end{align*}
Substituting the formula \eqref{reproducing} into the above identity, we get the desired formula \eqref{variance}.
\end{proof}

\begin{lem}\label{lem-app}
Let $K$ be a kernel satisfying Conditions \ref{con1} and \ref{con2}.  Then for any fixed $R > 0$, there exists a sequence  $(\varphi_n)_{n\in \N}$ of real-valued Schwartz functions,  such that $|\varphi_n(x)|\le 1$ and
$$ 
\lim_{n\to \infty} \sup_{x\in [-R, R]}  |  \varphi_n(x) - 1 | =0 \an \lim_{n\to \infty} \Var_{\PP_K}(T[\varphi_n]) = 0.
$$
\end{lem}

\begin{proof}
It suffices to prove that given any positive integer $n\in \N$, we can construct a  real-valued Schwartz function $\varphi_n$ such that 
$$
|\varphi_n(x)|\le 1,  \quad
\sup_{x\in [-R, R]}  |  \varphi_n(x) - 1 | \le 1/n \an  \Var_{\PP_K}(T[\varphi_n]) \le 1/n.
$$

Let $M>0$ be the number given in Condition \ref{con2}.  Fix a real number $N>1$ which will be specified later. Given a real-valued Schwartz function $f$, denote 
$$
F(x,y):  = \frac{1}{2} | f(x) - f(y) |^2 | K(x,y)|^2.
$$
We define $I_i(f), i = 1, 2, 3, 4$ as follows: 
\begin{align}
\begin{split}
\Var_{\PP_K}(T[f])  \le &  \underbrace{\iint_{|x| \le N M, | y| \le NM} F }_{=: I_1(f)}+\underbrace{\iint_{|x| \le M, | y| \ge NM} F }_{=: I_2(f)} + 
\\
& + \underbrace{\iint_{|x| \ge NM, | y| \le M} F }_{=: I_3(f)} + \underbrace{\iint_{|x| \ge M, | y| \ge M} F }_{=: I_4(f)}.
\end{split}
\end{align} 

\medskip
{\bf Step 1:} Control of $I_2$ and $I_3$. 

We claim that
\begin{align}\label{off-diag}
\lim_{N \to \infty} \iint_{\{| x | \le M, | y | \ge N \cdot M\}} | K(x,y)|^2 dxdy = 0.
\end{align}
Indeed, by Lemma \ref{lem-repro} and condition \eqref{no-singularity}, we have
$$
\iint_{\{| x | \le M, | y | \ge N \cdot M\}} | K(x,y)|^2 dxdy \le \int_{| x| \le M}  K(x,x) dx<\infty.
$$
Then the claim in \eqref{off-diag} follows from above inequality and bounded convergence theorem. Now let us choose $N\ge R +1$ large enough, such that 
$$
\iint_{\{| x | \le M, | y | \ge NM\}} | K(x,y)|^2 dxdy  = \iint_{\{| x | \ge N M, | y | \le M\}} | K(x,y)|^2 dxdy \le \frac{1}{40n}.
$$
It follows that for any function $f$ such that $|f| \le1$, we have
$$
I_2(f) + I_3(f) \le \frac{1}{10n}.
$$

In what follows, we fix $N$ chosen as above.

\medskip

 {\bf Step 2:} Control of $I_1$.
 
Note that $N$ being fixed,  the number $NM$ is also fixed.  For any function $f$, we have
 \begin{align*}
 I_1(f)& \le \Big(\sup_{|x| \le NM, | y| \le NM} | f(x)-f(y)|^2\Big)  \iint_{|x| \le NM, | y| \le NM}   | K(x,y)|^2dxdy
 \\ 
 & \le 4  \Big(\sup_{|x| \le NM} | f(x)-1|^2\Big)  \int_{|x| \le NM}   | K(x,x)|dx.
 \end{align*}
 It follows that for any $f$ such that
 $$
\sup_{|x| \le NM} | f(x)-1| \le   \min \Big\{n^{-1},   \left( 20 n \cdot  \int_{|x| \le NM}   K(x,x)dx \right)^{-1/2}\Big\},
 $$
 we have
 $$
 I_1(f) \le \frac{1}{10n}.
 $$
For future use, let us denote 
\begin{align}\label{delta-n}
\delta_n: = \min \Big\{n^{-1},   \left( 20 n \cdot  \int_{|x| \le NM}   K(x,x)dx \right)^{-1/2}\Big\}.
\end{align}
\medskip

{\bf Step 3:} Control of $I_4$.

By \eqref{pt-es}, we may write 
\begin{align}\label{f-trans}
\begin{split}
I_4(f) &\le  \iint_{\R^2} | f(x) - f(y) |^2 \Phi(x-y) dxdy
\\ & = \iint_{\R^2} | f(x+t) - f(x)|^2 \Phi(t) dxdt
\\ & = \iint_{\R^2}  |\widehat{f}(\xi) |^2 | e^{i 2\pi t\xi} - 1|^2 \Phi(t) d\xi dt 
\\ & =  \int_\R |\widehat{f}(\xi) |^2 ( 2 \widehat{\Phi}(0) - \widehat{\Phi} (\xi) - \widehat{\Phi} (- \xi)) d\xi,
\end{split}
\end{align}
where $\widehat{f}$ and $\widehat{\Phi}$ are the Fourier transforms of $f$ and $\Phi$ respectively.
Now we will apply a result from \cite{Boas}, which says that for a positive integrable function $\Phi$, condition \eqref{tail-condition} is equivalent to 
\begin{align}\label{lip}
\widehat{\Phi} (\zeta + \xi) + \widehat{\Phi} (\zeta  - \xi) - 2 \widehat{\Phi} (\zeta ) = O(| \xi|), \text{\, uniformly in $\zeta$, as $| \xi| \to 0$.}
\end{align}
Take $\zeta=0$ in \eqref{lip} and note that $\widehat{\Phi}$ is bounded, we see that there exists a numerical constant $C$  which only depends on $\Phi$, such that
$$
2 \widehat{\Phi}(0) - \widehat{\Phi} (\xi) - \widehat{\Phi} (- \xi) \le C|\xi|, \text{\, for all $\xi \in \R$.}
$$
 Substitue this inequality into the estimate \eqref{f-trans}, we obtain
$$
I_4(f) \le C \int_\R  | \xi| | \widehat{f}(\xi) |^2 d\xi.
$$

\medskip

{\bf Step 4:} Construction of $\varphi_n$.

Recall the definition of $\delta_n$ in \eqref{delta-n}.  Let $k\ge n$ be large enough such that for any $|t| \le N M k^{-1}$, we have 
$$
| e^{i 2 \pi t } -1  |\le \delta_n.
$$
 We claim that there exists a non-negative even function $\psi_n \in C_c^\infty(\R)$ supported in a $(\frac{1}{k})$-neighbourhood of $0$, such that 
\begin{align}\label{psi-n}
 \int_\R \psi_n(\xi) d\xi = 1 \an \int_\R | \xi|   \psi_n(\xi)^2 d\xi \le \frac{1}{10 C n}.
 \end{align}
 Indeed, since the function $\frac{1}{10 Cn |\xi|} \chi_{| \xi| \le 1/k}$ is not integrable, we can easily construct a Schwartz function $\psi_n$ such that  
 $$
 \int_\R \psi_n = 1 \an \psi_n(\xi) \le \frac{1}{10Cn |\xi|} \chi_{| \xi| \le 1/k}, \text{\, for any $\xi \in \R$}.
 $$
 This last pointwise inequality implies that $\supp (\psi_n) \subset [ - 1/k, 1/k]$ and 
 $$
  \int_\R | \xi|   \psi_n(\xi)^2 d\xi \le  \Big(\sup_{\xi  } | \xi| \psi_n(\xi)\Big)  \cdot \int_\R \psi_n(\xi)d\xi \le \frac{1}{10 Cn}.  
 $$
 Now set 
 $$
 \varphi_n (x) = \check{\psi}_n (x) = \int_\R \psi_n(\xi) e^{i 2 \pi x \xi} d\xi.
 $$ 
 Then  $\varphi_n \in \mathscr{S}(\R)$,  $\varphi_n(0) = 1$ and $| \varphi_n(x)| \le 1$. Moreover, since $\psi_n$ is an even real-valued function, $\varphi_n$ is real-valued.  By construction, we have
 $$
 I_4(\varphi_n) \le C \int_\R | \xi| | \widehat{\varphi}_n(\xi)|^2 d\xi \le \frac{1}{10n}.
 $$
Moreover, by our choice of $k$, we know that if $| \xi | \le k^{-1}$ and $|x|\le NM$, then we have $| e^{i 2 \pi x \xi} -1  |\le \delta_n$. Hence for any $| x| \le NM$, 
\begin{align*}
|  \varphi_n(x) - 1 |   &=  |  \varphi_n(x) - \varphi_n(0) | \le \int_\R | e^{i 2 \pi x \xi} -1  | |\psi_n (\xi) |d\xi 
\\
& = \int_{| \xi | \le k^{-1}} | e^{i 2 \pi x \xi} -1  | |\psi_n (\xi) |d\xi\le \delta_n \le n^{-1}.
\end{align*}
By Step 2, the above inequality implies that $I_1(\varphi_n) \le \frac{1}{10n}$. It is readily seen that we also have  $I_i(\varphi_n) \le \frac{1}{10n}, i = 2,3$, hence
$$
\Var_{\PP_K}(T[\varphi_n]) \le \sum_{i =1}^4 I_i(\varphi_n) \le \frac{1}{n}.
$$
This completes the proof of the proposition.
\end{proof}

\begin{rem}
The construction in \eqref{psi-n} relies heavily on the non-integrability of $\frac{1}{| \xi|}$ in any neighbourhood of the origin. Indeed, given a positive function $a(\xi)$, 
$$
\left(\int_\R a(\xi)^{-1} d\xi \right)^{-1} =  \inf \left\{ \int_\R a(\xi) \psi(\xi)^2 d\xi:  \text{$\psi$ positive and $\int_\R \psi(\xi) d\xi = 1$} \right\},
$$
with the understanding that the left hand side equals to 0 if $a(\xi)^{-1}$ is not integrable.
\end{rem}

Now we can prove Theorem \ref{intro-thm-rigid}. Our proof follows the line of that of \cite[Thm. 6.1]{Ghosh-rigid}.

\begin{proof}[Proof of Theorem \ref{intro-thm-rigid}]
Let $B\subset \R$ be any bounded measurable subset. Choose $R>0$ large enough such that $B \subset [-R, R]$. Let $\varphi_n$ be a sequence of Schwartz functions constructed as in Lemma \ref{lem-app}. We have 
\begin{align*}
& T[\varphi_n] (X) = \sum_{x\in X\cap B} \varphi_n(x) \sgn(x) + \sum_{x\in X\setminus B} \varphi_n(x) \sgn(x) =: I(n) + II(n).
\end{align*}
First note that 
\begin{align*}
& \| I(n) -  \sum_{x\in X\cap B} \sgn(x) \|_1  \le   \E_{\PP_K} \sum_{x\in X} |1- \varphi_n(x)| \chi_B(x) 
\\
& \le \sup_{x\in [-R, R]}   | 1 - \varphi_n(x) |  \cdot  \int_BK(x,x)dx,
\end{align*}
we have, passing to  a subsequence if necessary, 
\begin{align}\label{I(n)}
I(n) \xrightarrow[\text{$\PP_K$-almost surely}]{n \to \infty}  \#_{B\cap \R_{+}} -  \#_{B\cap \R_{-}}.
\end{align} 
By construction, $\lim_{n\to \infty}\Var_{\PP_K}(T[\varphi_n]) =0$, passing to  a subsequence if necessary, we have 
\begin{align}\label{II(n)}
I(n) + II(n) - \E_{\PP_K} T[\varphi_n] \xrightarrow[\text{$\PP_K$-almost surely}]{n \to \infty} 0.
\end{align}
Combining \eqref{I(n)} and \eqref{II(n)},  for $\PP_K$-almost every configuration $X\in \Conf(\R)$, we get
\begin{align*}
\#_{B\cap \R_{+}} (X) -  \#_{B\cap \R_{-}} (X) = \lim_{n\to \infty}\left( \E_{\PP_K} T[\varphi_n] - II(n) (X)\right).
\end{align*}
Since all the functions $ \E_{\PP_K} T[\varphi_n] - II(n)$ are $\mathcal{F}_{\R \setminus B}$-measurable, the $\PP_K$-almost sure limit function $\#_{B\cap \R_{+}} -  \#_{B\cap \R_{-}} $ is $\mathcal{F}_{\R \setminus B}^{\PP_K}$-measurable. The proof of Theorem \ref{intro-thm-rigid} is complete.
\end{proof}

\subsection{Proof of Theorem A}

Following \cite[Thm. 5.3]{BO-hyper}, when $z, z'$ are fixed, we denote the Whittaker kernel $\mathcal{K}_{z,z'}$ simply by $\mathcal{K}$.  The change of variables $x \mapsto 1/x$ transforms the Whittaker kernel to the following new kernel
\begin{align}\label{Knew}
K_\new(x,y) = \frac{1}{| xy|} \mathcal{K}(1/x, 1/y).
\end{align}
Note that the kernel $K_\new$ satisfies Condition \ref{con1}. Indeed, it is known in \cite{BO-hyper} that the Whittaker kernel $\mathcal{K}$ is such that  the operators $P_{+}\mathcal{K} P_{+}$ and $P_{-}\mathcal{K}P_{-}$ are non-negative and for any subsets $\Delta_1, \Delta_2$  both with positive distance from the origin such that $\Delta_1\subset \R_{+} $ and $\Delta_2 \subset \R_{-}$, the operators $\chi_{\Delta_i} \mathcal{K} \chi_{\Delta_i} (i = 1, 2)$ are in trace-class, and the operator $\chi_{\Delta_2}\mathcal{K}\chi_{\Delta_1}$ is Hilbert-Schmidt. Moreover the operator 
$$
\widehat{\mathcal{K}} := \sgn \cdot \mathcal{K} +  \chi_{\R_{-}}
$$
 defines an {\it orthogonal projection} on $L^2(\R, dx)$.  By the change of variable $x \mapsto 1/x$, these properties imply exactly that the kernel $K_\new$ satisfies Condition \ref{con1}.

\begin{lem}
Assume that the parameters $z, z'$ satisfy the conditions $z' = \bar{z}$ and $ z\in \C\setminus \R$. Then $K_\new$ satisfies Condition \ref{con2}.
\end{lem}

\begin{proof}

From the explicit formula \eqref{whittaker} for the Whittaker kernel, we see that the diagonal value $\mathcal{K}(x,x)$ is given by 
\begin{align}\label{W-diag}
\mathcal{K}(x,x) = \pm(\mathcal{P}_{\pm}'(|x|) \mathcal{Q}_{\pm}(|x|) - \mathcal{Q}_{\pm}'(|x|) \mathcal{P}_{\pm}(|x|)), 
\end{align}
the sign $\pm$ depends on the sign $\sgn(x)$ of the real number $x\in \R^*$. Since the Whittaker function converges to 0 exponentially fast at infinity, it is readily seen that 
\begin{align}\label{W-inf}
\int_{|x|>\delta} \mathcal{K}(x,x)dx <\infty.
\end{align}
This in turn implies the condition \eqref{no-singularity} for $K_\new$ around the origin.

As in the proof of \cite[Prop. 4.1.3]{PII}, for $x > 0$ near the origin, by expressing the Whittaker functions in terms of confluent hypergeometric functions, the functions $\mathcal{P}_{+}$ and $\mathcal{Q}_{+}$ can be written as
\begin{align}\label{PQ-exp}
\begin{split}
\mathcal{P}_{+}(x) & =  x^{\frac{z-z'}{2}} A_1(x) + x^{\frac{z'-z}{2}} B_1(x)
\\
\mathcal{Q}_{+}(x) & =  x^{\frac{z-z'}{2}} A_2(x) + x^{\frac{z'-z}{2}} B_2(x)
\end{split}
\end{align}
where $A_i(x), B_i(x), i  = 1, 2,$ are analytic in a neighbourhood of the origin. At infinity, both functions tend to 0. Hence $\mathcal{P}_{+}$ and $\mathcal{Q}_{+}$ are bounded on $\R_{+}$. The fact that $\mathcal{P}_{-}$ and $\mathcal{Q}_{-}$ are bounded on $\R_{-}$ can be proved similarly. It follows that there exists $C>0$, such that  
$$
\text{for any $x, y \in \R^*$, \,} |\mathcal{K}(x,y) | \le \frac{C}{| x-y|},
$$
or,  equivalently,
\begin{align}\label{large-arg}
 \text{for any $x, y \in \R$, \,}  |K_\new (x,y) | \le \frac{C}{| x-y|}.
\end{align}
Now fix $M>0$, let $\delta'$ be a fixed number such that $0< \delta' <M/4$. We claim that there exists $C'>0$, such that 
\begin{align}\label{small-claim}
\text{if $|x|\ge  M, | y|\ge M$ and $| x-y| \le \delta'$,  then\,}  |K_\new(x,y)|\le C'.
\end{align}
Indeed, by the choice of $\delta'$, any pair $(x,y)$ verifying the hypothesis in \eqref{small-claim} satisfies $\sgn(x) \sgn(y)>0$. If $x>0, y>0$, then 
\begin{align}\label{near-diag}
\begin{split}
K_\new(x,y) & = \frac{\mathcal{P}_{+}(\frac{1}{x}) \mathcal{Q}_{+}(\frac{1}{y}) - \mathcal{Q}_{+}(\frac{1}{x}) \mathcal{P}_{+}(\frac{1}{y})}{y-x}  
\\ 
& = -  \mathcal{P}_{+}\left(\frac{1}{x}\right) \mathcal{Q}_{+}'\left(\frac{1}{\xi_{x,y}}\right) \frac{1}{\xi_{x,y}^2} + \mathcal{Q}_{+}\left(\frac{1}{x}\right) \mathcal{P}_{+}'\left(\frac{1}{\xi_{x,y}}\right) \frac{1}{\xi_{x,y}^2},
\end{split}
\end{align}
where $\xi_{x,y} \in (\min(x,y), \max(x,y))$. By \eqref{PQ-exp}, it is readily seen that 
\begin{align}\label{asymp-de}
\mathcal{P}_{+}'(1/x) = O(x), \quad \mathcal{Q}_{+}'(1/x) = O(x), \text{\, as $x \to \infty$.}
\end{align}
From  \eqref{near-diag} and \eqref{asymp-de}, it is readily seen that \eqref{small-claim} holds for $x>0,y>0$. Similarly, by analyzing $\mathcal{P}_{-}, \mathcal{Q}_{-}$, we also obtain \eqref{small-claim} for $x<0,y<0$.  Combining \eqref{large-arg} and \eqref{small-claim}, we see that the condition \eqref{pt-es} in Condition \ref{con2} holds for $K_\new$, that is 
$$
\text{if $|x|\ge  M, | y|\ge M$,  then\,} |K_\new(x,y)|^2 \le \Phi (x-y), 
$$ 
where
 $$\Phi (t) =   (C')^2 \chi_{|t| \le \delta'} + \frac{C^2}{t^2} \chi_{|t| \ge \delta'},$$  is a function satisfying the required condition \eqref{tail-condition}.
\end{proof}

\section{Balanced Palm equivalence property}\label{tolerance-sec}

Recall that for a point process $\PP$ on $\mathscr{E}$ and a positive integer $k\in \N$, the $k$-th correlation measure $\rho_{k}$ of $\PP$ is a positive measure on $\mathscr{E}^k$, which is defined by the relation \eqref{cor-measure}.

\begin{defn}\label{defn-Palm-equivalence}
A point process $\PP$ on $\mathscr{E}$ is said to have {\it balanced Palm equivalence property} with respect to the partition   $\mathscr{E}= \mathscr{E}_1 \sqcup \mathscr{E}_2$, if for any positive integer $n\in \N$, for $\rho_{2n}$-almost every $2n$-tuple $\mathfrak{p}  \in \mathscr{E}_1^n \times \mathscr{E}_2^n$ of distinct points, in other words, $\mathfrak{p}$ is a $2n$-tuple of distinct points of $\mathscr{E}$ with a equal number of points from $\mathscr{E}_1$ and $\mathscr{E}_1$, the Palm measure $\PP^{\mathfrak{p}}$ is equivalent to $\PP$. 
\end{defn}

For processes governed by  $J$-Hermitian 
kernels, the balanced Palm equivalence property is the natural analogue 
of equivalence of Palm measures of the same order  for processes with Hermitian kernels.

\subsection{Palm measures of \texorpdfstring{$L$}{a}-processes}\label{sec-L-rel}

In this section, we will study the correlation kernels of Palm measures for $L$-processes. 

Recall the definition of $L$-kernels in Definition \ref{defn-L-kernel}. Let  $\mu_L$ be the determinantal measure induced by a kernel $K_L = L ( 1  +L)^{-1}$, where $L$ is a kernel satisfying Condition \ref{con35}. By Lemma \ref{lem-V-intro},  the kernel $K_L$ is $J$-Hermitian.  When $L$ is fixed, we simply write $K = K_L$. Shirai-Takahashi's Theorem \ref{ST} says that for almost every $p \in \R^*$ (with respect to the measure $K(x,x)dx$),  the Palm measure $\mu_L^{p}$ is a determinantal point process with the following kernel:
\begin{align}\label{kp-kernel}
K^p(x,y) = K(x,y) - \frac{K(x, p)K(p,y)}{K(p,p)} = K(x,y) - \sgn(p) \sgn(y) \frac{K(x,p) \overline{K(y,p)}}{K(p,p)}.
\end{align}
Let $p^{+}>0$ and $p^{-}< 0$, our aim is  to describe the correlation kernel of the Palm measure $\mu_L^{(p^{+}, p^{-})}$, that is, the kernel 
$$
K^{(p^{+}, p^{-})} : = (K^{p^{+}})^{p^{-}} = (K^{p^{-}})^{p^{+}}.
$$
More generally,  we are going to describe the kernel $K^{\mathfrak{p}}$ defined by the formula \eqref{Kq-iteration} when $\mathfrak{p} =  (p^{+}_1, \dots p^{+}_n;  p^{-}_1, \dots, p^{-}_n)$ with $p_i^{+} > 0$ and $p^{-}_i <0$ for $i = 1, \dots, n$.  While it is easily seen that $K^{\mathfrak{p}}$ is a $J$-Hermitian kernel satisfying Condition \ref{con1} (transformed to the version with singularity at origin), it is a priori not clear whether $K^{\mathfrak{p}}$ admits an $L$-kernel. We now check that it does and that the  $L$-kernel of $K^{\mathfrak{p}}$ also satisfies Condition \ref{con35}.

\begin{defn}
Given  $\mathfrak{p}  = (p^{+}, p^{-})$, where $p^{+}> 0, p^{-}<0$, we define a {\it bounded} function on $\R^*$ by the formula
\begin{align}\label{defn-gp-1}
g_{\mathfrak{p}} (x) =  \frac{x-p^{+}}{x-p^{-}} \chi_{\{x > 0\}} + \frac{x-p^{-}}{x-p^{+}} \chi_{\{x<0\}}.
\end{align}
More generally, if $\mathfrak{p} = (p^{+}_1, \dots p^{+}_n;  p^{-}_1, \dots, p^{-}_n)$ with $p_i^{+} > 0$ and $p^{-}_i <0$ for $i = 1, \dots, n$, we set 
\begin{align}\label{gp}
g_{\mathfrak{p}}(x) = \prod_{i=1}^n \left(\frac{x-p_i^{+}}{x-p_i^{-}} \chi_{\{x > 0\}} + \frac{x-p_i^{-}}{x-p_i^{+}} \chi_{\{x<0\}}\right).
\end{align}
\end{defn}

\begin{prop}\label{prop-palm-kernel}
Let $L$ be an operator satisfying Condition \ref{con35}. If $\mathfrak{p}  = (p^{+}, p^{-})$ such that $p^{+}> 0, p^{-}<0$, then we have 
$$
K_L^{\mathfrak{p}} = K_{g_{\mathfrak{p}} L  g_{\mathfrak{p}}}.
$$
\end{prop}

\begin{cor}\label{cor-palm-kernel}
Let $L$ be an operator satisfying Condition \ref{con35}. Let 
$$
\mathfrak{p} = (p^{+}_1, \dots p^{+}_n;  p^{-}_1, \dots, p^{-}_n)
$$ be a $2n$-tuple of real numbers such that $p_i^{+} > 0$ and $p^{-}_i <0$ for $i = 1, \dots, n$,  then 
$$
K_L^{\mathfrak{p}} = K_{g_{\mathfrak{p}} L  g_{\mathfrak{p}}}.
$$
\end{cor}

\begin{proof}
When $\mathfrak{p} = (p_1^{+}, p_1^{-})$, this is just Proposition \ref{prop-palm-kernel}. Now since the new kernel 
$$
g_{(p_1^{+}, p_1^{-})}(x) L (x,y) g_{(p_1^{+}, p_1^{-})}(y)
$$
has a similar structure as $L(x,y)$, that is, it satisfies Condition \ref{con35}, we can continue our procedure and complete the proof of the corollary. 
\end{proof}

\begin{proof}[Proof of Proposition \ref{intro-Palm-prop}]
Note that in Corollary \ref{cor-palm-kernel}, we obtain $K_L^{\mathfrak{p}} = K_{g_{\mathfrak{p}} L  g_{\mathfrak{p}}}$. However, by the special form of $L$, we have (see Lemma \ref{lem-conjugation} below for this fact)
$$
g_{\mathfrak{p}} L  g_{\mathfrak{p}} = f_{\mathfrak{p}} L  f_{\mathfrak{p}},
$$
where $g_{\mathfrak{p}}$ and $f_{\mathfrak{p}}$ are functions defined in \eqref{gp} and \eqref{fp} respectively. Hence we obtain that 
$$
\mu_{L}^{\mathfrak{p}} = \mu_{f_\mathfrak{p} L f_{\mathfrak{p}}}.
$$
\end{proof}

\begin{lem}\label{lem-in-out}
Let $p \in \R^*$. Then the kernel $K^p(x,y) $ defined in \eqref{kp-kernel} is $J$-Hermitian and $\widehat{K^{p}}$ is an orthogonal projection. Moreover, if $p^{+}>0$, then 
$$
\Ran (\widehat{K^{p^{+}}}) =  \Ran(\widehat{K})  \ominus \C \sgn(\cdot) K(\cdot, p^{+});
$$ 
if $p^{-}< 0$, then 
$$
\Ran (\widehat{K^{p^{-}}}) =  \Ran(\widehat{K})  \oplus \C \sgn(\cdot) K(\cdot, p^{-}).
$$
\end{lem}

\begin{proof}
It is clear that the kernel $K^p(x,y) $ is $J$-Hermitian. By Lemma \ref{lem-repro}, we see that, for any $p \in \R^*$, the following kernel
$$ 
\ell_p(x,y) = \sgn(x) \sgn(y) \frac{K(x,p)\overline{ K(y,p)}}{K(p,p)}
$$
induces the  orthogonal projection  onto the one dimensional subspace $\C \sgn(\cdot) K(\cdot, p)$. Let us denote this one dimensional projection again by $\ell_p$. By definition, it is easy to see that 
$$
\widehat{K^p}  = \widehat{K} - \sgn(p)   \ell_p.
$$
That is, if $p^{+}>0$, then $\widehat{K^{p^{+}}}  = \widehat{K} -  \ell_{p^{+}}$ and if $p^{-}<0$, then $\widehat{K^{p^{-}}}  = \widehat{K} +  \ell_{p^{-}}$. Thus for proving Lemma \ref{lem-in-out}, we only need to show that 
\begin{align}\label{in-out}
\sgn(\cdot) K(\cdot, p^{+}) \in \Ran (\widehat{K}) \an \sgn(\cdot) K(\cdot, p^{-}) \in \Ran (\widehat{K})^\perp.
\end{align}
The first relation in \eqref{in-out} is equivalent to 
\begin{align}\label{in}
\int_\R K(x, y) \sgn(y) K(y, p^{+}) dy + \chi_{\R_{-}} (x) K(x,p^{+}) = K(x, p^{+});
\end{align}
while the second is equivalent to 
\begin{align}\label{out}
\int_\R K(x, y) \sgn(y) K(y, p^{-}) dy + \chi_{\R_{-}} (x) K(x,p^{-}) =0.
\end{align}

By using the fact that $\widehat{K}^2 = \widehat{K}$ and comparing all the block coefficients of the operator $\widehat{K}^2$ and $\widehat{K}$, both written in the block form as in \eqref{block-KK}, we get  
$$
\left\{\begin{array}{cc} K_{++} =  K_{++}^2 - K_{+-}K_{-+} \vspace{2mm}  \\ K_{--}  = K_{--}^2 - K_{-+} K_{+-} \vspace{2mm}\\ K_{++}K_{+-} = K_{+-} K_{--} \vspace{2mm} \\ K_{-+} K_{++}  = K_{--} K_{-+}  \end{array} .\right.
$$
The above first identity implies \eqref{in} for $x>0$; the second one implies \eqref{out} for $x<0$; the third one implies \eqref{out} for $x>0$ and the last one implies \eqref{in} for $x<0$. 
\end{proof}

\begin{rem}\label{quasi-repro}
Although $\widehat{K}$ is not the reproducing kernel of the Hilbert subspace $\Ran(\widehat{K})$, the space $\Ran(\widehat{K})$ still possesses certain reproducing feature. Indeed, if $\varphi \in \Ran(\widehat{K})$, then we have the following identity of functions in $L^2(\R)$:
\begin{align*}
\chi_{\R_{+}} (x) \varphi (x) & =   \int_\R K(x,y) \varphi(y) dy = \int_\R \sgn(y) \overline{K(y,x)} \varphi(y)dy \\ & = \langle\varphi, \sgn(\cdot) K(\cdot, x) \rangle_{L^2(\R)}.
\end{align*}
 \end{rem}

Now we can apply Lemma \ref{lem-in-out} to $K^{p^{+}}$ and $K^{p^{-}}$ respectively and get the following 
\begin{prop}
Let $\mathfrak{p}  = (p^{+}, p^{-})$ with $p^{+}> 0, p^{-}<0$, then 
\begin{align}\label{palm1}
\Ran (\widehat{K^{\mathfrak{p}}}) & = \Big(\Ran (\widehat{K}) \ominus \C \sgn(\cdot) K(\cdot, p^{+})\Big) \oplus  \C  \sgn(\cdot) K^{p^{+}} (\cdot, p^{-})\\ \label{palm2} & = \Big(\Ran (\widehat{K}) \oplus \C \sgn(\cdot) K(\cdot, p^{-}) \Big) \ominus \C \sgn(\cdot) K^{p^{-}} (\cdot, p^{+}).
\end{align}
\end{prop}

We also need an explicit description of subspaces as $\Ran (\widehat{K_L})$. It is convenient for us to  introduce a general condition on the kernel $L$.

\begin{con}\label{con-four}
The $L$-operator is assumed to have the following block form with respect to the decomposition $L^2(\R^*, dx) = L^2(\R_{+}, dx) \oplus L^2(\R_{-}, dx)$: 
\begin{align}\label{L-block}
L = \left[  \begin{array}{cc} 0 & V \\  - V^* & 0 \end{array}\right],
\end{align}
where $V: L^2(\R_{-}) \rightarrow L^2(\R_{+})$ is a bounded  linear operator. Moreover, assume that the operator $V$ is such that for any $\varepsilon > 0$, the operators $\chi_{(\varepsilon, \infty)} V$ and  $V \chi_{(-\infty,-\varepsilon)} $ are Hilbert-Schmidt.  
\end{con}

$L$-kernels satisfying Condition \ref{con-four} appear naturally in many contexts, see e.g. \cite{BOO-jams}.

\begin{prop}[\cite{BO-R}, Prop. 5. 1]\label{prop-proj}
Let $L$ be an operator as in Condition \ref{con-four}, then the operator $\widehat{K}_L$ is an orthogonal projection, the range $\Ran(\widehat{K}_L) $ and its orthogonal complement are given by 
\begin{align}
\label{range}
\Ran(\widehat{K}_L)& =  \Big\{Vh \oplus  h: h\in L^2(\R_{-}, dx)  \Big\};
\\
\label{ortho}
\Ran(\widehat{K}_L)^\perp & =  \Big\{f \oplus  (- V^*f): f \in L^2(\R_{+}, dx)  \Big\}.
\end{align}
\end{prop}

\begin{rem}\label{rem-fLf}
Let $f: \R^* \rightarrow\C$ be a bounded function, then 
$$
fLf = \left[  \begin{array}{cc} 0 & f \chi_{\R_{+}} V    f \chi_{\R_{-}}  \\  -  f\chi_{\R_{-}}V^* f \chi_{\R_{+}}  & 0 \end{array}\right] =  \left[  \begin{array}{cc} 0 & f^{+} V    f^{-}   \\  -  f^{-}V^* f^{+}  & 0 \end{array}\right].  
$$
\end{rem}

\begin{proof}[Proof of Proposition \ref{prop-palm-kernel}]
Since $K \leftrightarrow \widehat{K}$ is a bijection,  to show $K_L^{\mathfrak{p}} = K_{g_{\mathfrak{p}} L  g_{\mathfrak{p}}}$ is equivalent to show the coincidence of two orthogonal projections:
\begin{align}\label{KLp-KgLg}
\widehat{K_L^{\mathfrak{p}}}  = \widehat{K_{g_{\mathfrak{p}} L  g_{\mathfrak{p}}}}.
\end{align}
By Proposition \ref{prop-proj} and Remark \ref{rem-fLf}, we have 
\begin{align*}
 \Ran( \widehat{K_{g_{\mathfrak{p}} L  g_{\mathfrak{p}}}} ) & = \Big\{ g_{\mathfrak{p}}^{+}V(g_{\mathfrak{p}}^{-} h)  \oplus h : h\in L^2(\R_{-}, dx)  \Big\}.
\end{align*}
Hence to show the identity \eqref{KLp-KgLg},  it suffices to show  the coincidence of the following two subspaces 
$$
\Ran(\widehat{K_L^{\mathfrak{p}}}  ) = \Big(\Ran (\widehat{K_L}) \ominus \C \sgn(\cdot) K_L(\cdot, p^{+})\Big) \oplus  \C  \sgn(\cdot) K_L^{p^{+}} (\cdot, p^{-})
$$ 
and 
\begin{align}\label{range-hat-KgLg}
 \Ran( \widehat{K_{g_{\mathfrak{p}} L  g_{\mathfrak{p}}}} ) = \Big\{ g_{\mathfrak{p}}^{+}V(g_{\mathfrak{p}}^{-} h)  \oplus h : h\in L^2(\R_{-}, dx)  \Big\}.
\end{align}

\medskip

{\bf Step 1}: 
If $\varphi \in \Ran (\widehat{K_L}) \ominus \C \sgn(\cdot) K_L(\cdot, p^{+})$, then $ \varphi\in \Ran( \widehat{K_{g_{\mathfrak{p}} L  g_{\mathfrak{p}}}} ) $. 

\medskip
Recall that 
$$
\Ran(\widehat{K}_L) =  \Big\{Vh \oplus  h: h\in L^2(\R_{-}, dx)  \Big\}. 
$$
Hence the hypothesis $\varphi \in \Ran (\widehat{K_L}) \ominus \C \sgn(\cdot) K_L(\cdot, p^{+})$ is equivalent to the existence of a functioin $h\in L^2(\R_{-})$ such that 
$$
\varphi = h + V(h) \an \varphi \perp \sgn(\cdot) K_L(\cdot, p^{+}).
$$
By Remark \ref{quasi-repro}, this last condition can be translated to the condition $V(h)(p^{+}) =0$, that is 
$$
A(p^{+}) \int_{\R_{-}} \frac{A(y) h(y)}{p^{+} - y} dy =0.
$$
Since $A$ is assumed to have full support, the set $\{p\in \R^*: A(p) =0\}$ is negligible, hence we may assume that $A(p^{+}) \ne 0$. Thus we have 
$$
\int_{\R_{-}} \frac{A(y) h(y)}{p^{+} - y} dy =0.
$$
Now we want to show that there exists $h_1\in L^2(\R_{-})$ such that 
$$
h + V(h) = h_1 + g_{\mathfrak{p}}^{+} V (h_1g_{\mathfrak{p}}^{-}).
$$
The above identity is equivalent to 
$$
h  = h_1 \an V(h) = g_{\mathfrak{p}}^{+} V (h_1g_{\mathfrak{p}}^{-}).
$$
Hence what we need to show is:  once we have $V(h)(p^{+}) = 0$, then 
$$
V(h) = \frac{x-p^{+}}{ x - p^{-}} V  ( \frac{x-p^{-}}{ x - p^{+}} h).
$$
The above assertion is equivalent to 
 \begin{align}\label{pre-commutator}
\frac{1 }{ x - p^{+}} V(h)(x) = V  ( \frac{1}{ x - p^{+}} h)(x).
\end{align}
If we denote $k = \frac{h}{x-p^{+}}$, then the identity \eqref{pre-commutator} is equivalent to 
\begin{align}\label{commutator}
[x, V] k = 0,
\end{align}
where $[x, V]$ is the commutator between the multiplication $x$ and $V$. Since the commutator $[x, V]$ has a kernel given by the formula 
 $ \chi_{\R_{+}} (x) A(x)A(y) \chi_{\R_{-}} (y)$, hence the identity \eqref{commutator} can be checked as follows: 
 $$
 ([x, V] k)(x) = \chi_{\R_{+}} (x) A(x) \int_{\R_{-}} A(y) k(y)dy = \chi_{\R_{+}} (x) A(x)\int_{\R_{-}} \frac{A(y) h(y)}{ y - p^{+} } dy =0.
 $$

 \medskip
 
 {\bf Step 2}:  If $\varphi  = \sgn(\cdot) K_L^{p^{+}} (\cdot, p^{-})$, then $\varphi\in \Ran( \widehat{K_{g_{\mathfrak{p}} L  g_{\mathfrak{p}}}} ) $. 
 
 \medskip

By \eqref{range-hat-KgLg}, what we need to show is that
\begin{align}\label{boss}
\varphi (x) \chi_{\R_{+}} (x) =  \frac{x-p^{+}}{ x - p^{-}} V \left( \frac{x -p^{-}}{ x - p^{+}}   \varphi (x) \chi_{\R_{-}} (x) \right) (x).
\end{align}
This is in turn equivalent to the following assertion: for $x > 0$, we have 
\begin{align}\label{to-prove}
K_L^{p^{+}} (x, p^{-}) =  -  \frac{x-p^{+}}{ x - p^{-}} V \left( \frac{x -p^{-}}{ x - p^{+}}   K_L^{p^{+}} (\cdot, p^{-}) \right) (x).
\end{align}

By a result in \cite[Section II]{IIKS}, under Condition \ref{con35}, the kernel $K(x,y) = K_L(x,y)$ has the following integrable form 
$$
K(x,y) = \frac{F_1(x) G_1(y) + F_2(x) G_2(y)}{x-y}, 
$$
where 
$$
\left\{
\begin{array}{c} 
( 1 + L)F_1 = A^{+}
\\ 
( 1+L)F_2 = A^{-}
\\
(1+L^*) G_1 = A^{-}
\\
(1 + L^{*})G_2 = A^{+}
\end{array}.
\right.
$$
Note that $L = V - V^*$, and since $V, V^*$ has range in $L^2(\R_{+}), L^2(\R_{-})$ respectively, the above equation system is equivalent to 
\begin{align}\label{sep-rel}
\begin{split}
\left\{
\begin{array}{ll} 
(F_1)_{-} - V^* F_1 = 0, & (F_1)_{+} + VF_1 = A^{+}
\\
(F_2)_{-} - V^* F_2 = A^{-}, & (F_2)_{+} + VF_2 = 0
\\
(G_1)_{-} + V^* G_1 = A^{-}, & (G_1)_{+} - VG_1 = 0
\\
(G_2)_{-} + V^* G_2 = 0, & (G_2)_{+} - VG_2= A^{+}
\end{array}.
\right.
\end{split}
\end{align}
Moreover, we have 
\begin{align}\label{0}
F_1(x) G_1(x) + F_2(x) G_2(x) =0.
\end{align} From this, by  l'H\^opital's rule, we have 
\begin{align}\label{hopital}
K(x,x) = F_1'(x) G_1(x) + F_2'(x) G_2(x).
\end{align}
For $x > 0$ and $x \ne p^{+}$, we have 
\begin{align*}
&  \frac{1}{A^{+}(x)}V\left( \frac{F_1}{x-p^{+}} \right)(x) 
= \int_{\R_{-}}  \frac{A^{-}(y)  F_1(y) }{ (x-y)(y-p^{+})  } dy 
\\ & = \frac{1}{x - p^{+}}\int_{\R_{-}}  \left(\frac{A^{-} (y) F_1(y)  }{x-y}  -  \frac{A^{-} (y) F_1(y)  }{p^{+}-y} \right) dy 
\\ & = \frac{1}{x-p^{+}} \Big[ \frac{(VF_1)(x)}{A^{+}(x)} - \frac{(VF_1)(p^{+})}{A^{+}(p^{+})}  \Big] 
\\ & = \frac{1}{x-p^{+}} \Big(  \frac{ A^{+}(x) - F_1(x)}{A^{+}(x)}   - \frac{ A^{+}(p^{+}) - F_1(p^{+}) }{A^{+} (p^{+})}  \Big) 
\\ & = \frac{1}{x-p^{+}} \Big(  \frac{- F_1(x)}{A^{+}(x)}   + \frac{ F_1(p^{+}) }{A^{+} (p^{+})}  \Big).
\end{align*}
Similarly, if $x > 0$ and $x \ne p^{+}$, then 
\begin{align*}
& \frac{1}{A^{+} (x) }V\left( \frac{F_2}{x-p^{+}} \right)(x) =\frac{1}{x-p^{+}} \Big[ \frac{(VF_2)(x)}{A^{+}(x)} - \frac{(VF_2)(p^{+})}{A^{+}(p^{+})}  \Big] 
\\ & = \frac{1}{x-p^{+}} \Big(  \frac{- F_2(x)}{A^{+}(x)}   + \frac{ F_2(p^{+}) }{A^{+} (p^{+})}  \Big).
\end{align*}
 We thus get
\begin{align}\label{minus}
\begin{split}
& V\left( \frac{x -p^{-}}{ x - p^{+}}  K(x, p^{-})  \right) (x) =  - \frac{x-p^{-}}{x-p^{+}} K(x, p^{-})   + \frac{p^{+} - p^{-}}{x-p^{+}} \frac{K(p^{+}, p^{-}) }{A^{+} (p^{+})} A^{+} (x).
\end{split}
\end{align}

Now note that we have
\begin{align}\label{fraction-dev}
\frac{y-p^{-}}{ (x-y) (y-p^{+})^2} = \frac{x-p^{-}}{ (x-p^{+})^2}  \frac{1}{x-y} -  \frac{x-p^{-}}{ (x-p^{+})^2}  \frac{1}{p^{+}- y} +  \frac{p^{+}- p^{-}}{ x - p^{+}}\frac{1}{(p^{+} - y )^2}.
\end{align}
Given $f \in L^2(\R_{-},dx)$, we have 
\begin{align}\label{square}
& \int_{\R_{-}} \frac{A^{-}(y) f(y)}{(x-y)^2} dy = -\frac{d}{dx} \int_{\R_{-}} \frac{A^{-}(y) f(y)}{x-y} dy  =  -\frac{d}{dx} \left[\frac{V(f)}{A^{+}}\right](x).
\end{align}
Applying identities \eqref{fraction-dev} and \eqref{square} and by denoting 
$$
H(x) = F_1(x) G_1(p^{+}) +  F_2(x) G_2(p^{+}),
$$ 
we get 
\begin{align*}
& \frac{V\left( \frac{x -p^{-}}{ x - p^{+}}  K(x, p^{+})  \right) (x)}{A^{+}(x)} = \int_{\R_{-}} \frac{ (y-p^{-})A^{-}(y)  H(y) }{(x-y) (y-p^{+})^2} dy
\\ 
= &  \frac{x-p^{-}}{ (x-p^{+})^2}   \int_{\R_{-}}  \frac{A^{-}(y)  H(y) }{x-y}dy
 - \frac{x-p^{-}}{ (x-p^{+})^2}  \int_{\R_{-}} \frac{A^{-}(y)  H(y) }{p^{+}- y} dy 
\\
&+    \frac{p^{+}- p^{-}}{ x - p^{+}}  \int_{\R_{-}} \frac{A^{-}(y)  H(y) }{(p^{+} - y )^2} dy
\\
= & \frac{x-p^{-}}{ (x-p^{+})^2}  \Big[ \frac{V(H)(x)}{A^{+}(x)} - \frac{V(H) (p^{+})}{A^{+} (p^{+})} \Big]
 -  \frac{p^{+}- p^{-}}{ x - p^{+}}   \frac{d}{dx} \left[\frac{V(H)}{A^{+}}\right](p^{+}).
\end{align*}
By \eqref{sep-rel}, for $x> 0$, we have 
$$
\left\{
\begin{array}{l}
 \frac{d}{dx} \left(\frac{V(F_1)}{A^{+}}\right)(x) =  \frac{F_1'(x) A^{+}(x) - F_1(x) \frac{d}{dx}A^{+}(x)}{A^{+}(x)^2}
 \\
  \frac{d}{dx} \left(\frac{V(F_2)}{A^{+}}\right)(x) =  \frac{F_2'(x) A^{+}(x) - F_2(x)  \frac{d}{dx} A^{+}(x)}{A^{+}(x)^2}
\end{array}.
\right.
$$
Keeping in mind that  the identities \eqref{0} and \eqref{hopital} hold, we obtain 
\begin{align}\label{plus}
\begin{split}
V\left( \frac{x -p^{-}}{ x - p^{+}}  K(x, p^{+})  \right) (x)
  = &   - \frac{x-p^{-}}{ x-p^{+}} K(x, p^{+})  
+ \frac{p^{+} - p^{-}}{x-p^{+}}     \frac{K(p^{+}, p^{+})}{A^{+} (p^{+})} A^{+} (x).
\end{split}
\end{align}
Combining identities \eqref{minus} and \eqref{plus}, we get 
\begin{align}\label{final}
\begin{split}
& V\left( \frac{x -p^{-}}{ x - p^{+}} \Big[  K(p^{+}, p^{+}) K(x, p^{-}) - K(x, p^{+}) K(p^{+}, p^{-}) \Big]\right) (x)
\\ 
& = -     \frac{x-p^{-}}{x-p^{+}} \Big\{ K(p^{+}, p^{+}) K(x, p^{-}) - K(x, p^{+}) K(p^{+}, p^{-}) \Big\}. 
\end{split}
\end{align}

Since \begin{align*}
K(p^{+}, p^{+})  K^{p^{+}} (x, p^{-}) = K(p^{+}, p^{+}) K(x,p^{-}) - K(p^{+}, p^{-}) K(x, p^{+}),
\end{align*}
the identities \eqref{to-prove} and \eqref{final} are equivalent. Thus we complete the proof of Step 2.

\begin{rem}
Denote 
$$
  \psi(x):  = \frac{x -p^{-}}{ x - p^{+}} \varphi (x) =  \frac{x -p^{-}}{ x - p^{+}}   \sgn(\cdot) K_L^{p^{+}} (\cdot, p^{-}),
 $$ 
 then we can show that the identity  \eqref{boss} is equivalent to 
$
\psi_{+}=  K_L (\psi),
$
which is in turn equivalent to
 $ \psi = \widehat{K_L} (\psi).$
Hence in Step 2, we in fact proved that 
$$
 \frac{x -p^{-}}{ x - p^{+}}   \sgn(\cdot) K_L^{p^{+}} (\cdot, p^{-}) \in \Ran (\widehat{K_L}).
$$
\end{rem}

\medskip

{\bf Step 3:} Now we want to prove that $\Ran(\widehat{K_L^{\mathfrak{p}}}  ) \supset  \Ran( \widehat{K_{g_{\mathfrak{p}} L  g_{\mathfrak{p}}}} )$ , this is equivalent to
\begin{align}\label{reverse-inclusion}
\Ran(\widehat{K_L^{\mathfrak{p}}}  )^{\perp} \subset  \Ran( \widehat{K_{g_{\mathfrak{p}} L  g_{\mathfrak{p}}}} )^\perp.
\end{align}
By using \eqref{palm2}, we have 
$$
\Ran(\widehat{K_L^{\mathfrak{p}}}  )^{\perp}   = \Big(\Ran (\widehat{K_L})^\perp \ominus \C \sgn(\cdot) K_L(\cdot, p^{-}) \Big) \oplus \C \sgn(\cdot) K_L^{p^{-}} (\cdot, p^{+}).
$$
From \eqref{ortho}, we know that 
\begin{align*}
\Ran(\widehat{K_L})^\perp& =  \Big\{f_1 \oplus  (- V^*f_1): f_1 \in L^2(\R_{+}, dx)  \Big\},
\\
\Ran(\widehat{K_{g_{\mathfrak{p}} L  g_{\mathfrak{p}}}})^\perp &=  \Big\{f_2 \oplus  \Big(- g_{\mathfrak{p}}^{-}V^* (g_{\mathfrak{p}}^{+} f_2) \Big): f_2 \in L^2(\R_{+}, dx)  \Big\}.
\end{align*}
Now after switching the r\^oles of $L^2(\R_{-}, dx)$ and $L^2(\R_{+}, dx)$;  the kernels $V$ and $- V^*$; the pairs $(p^{+}, p^{-})$ and $(p^{-}, p^{+})$; the functions $\frac{x-p^{+}}{x-p^{-}}$ and $\frac{x-p^{-}}{x-p^{+}}$; and also the pairs of vectors $( \sgn(\cdot) K_L(\cdot, p^{+}), \sgn(\cdot) K_L^{p^{+}} (\cdot, p^{-}))$ and $( \sgn(\cdot) K_L(\cdot, p^{-}), \sgn(\cdot) K_L^{p^{-}} (\cdot, p^{+}))$, we arrive exactly at the same situation as above in proving $\Ran(\widehat{K_L^{\mathfrak{p}}}  ) \subset  \Ran( \widehat{K_{g_{\mathfrak{p}} L  g_{\mathfrak{p}}}} )$.  Hence we  may obtain \eqref{reverse-inclusion} by repeating the same arguments in Step 1 and Step 2.
\end{proof}

\subsection{Sufficient condition for equivalence of two \texorpdfstring{$L$}{a}-processes}\label{section-conditions}

In this section, we formulate a sufficient condition for two $L$-processes to be equivalent on the level of their $L$-kernels. 

\begin{lem}\label{lem-V}
Let $L$ be an operator satisfying Condition \ref{con-four}.  Then the operator $K_L$ is a $J$-self-adjoint operator satisfying all the conditions of Theorem \ref{thm-0}. In particular, $K_L$ is the correlation kernel of a determinantal point process on $\R^*$.
\end{lem}

\begin{lem}\label{lem-Kg}
Let $L$ be an operator satisfying Condition \ref{con-four}. Let $g: \R\rightarrow \R$ be a bounded Borel function. 
Then the operator $gLg$ satisfies also Condition \ref{con-four}. Moreover,  the operator $K_{gLg} = gLg ( 1 + gLg)^{-1}$ is given by 
\begin{align}\label{Kg-form}
K_{gLg} = gK_L (1 + (g^2-1)K_L)^{-1}g.
\end{align}
\end{lem}

\begin{con}\label{con5}
Assume that $L$ is a kernel given as in \eqref{L-block} such that the operator $V: L^2(\R_{-}) \rightarrow L^2(\R_{+})$ satisfying the following condition: for any $\varepsilon >0$, the operators  $\chi_{(\varepsilon, \infty)}V$, $V\chi_{(-\infty, -\varepsilon)}$,  $\chi_{(0, \varepsilon)} x V$ and  $V y\chi_{(-\varepsilon, 0)}$ are Hilbert-Schmidt.
\end{con}

\begin{lem}\label{4-5}
An $L$-kernel satisfying Condition \ref{con35} is an $L$-kernel satisfying Condition \ref{con5}. 
\end{lem}

Recall that we denote by $\mu_L$ the determinantal point process on $\R^*$ whose correlation kernel is $K_L = L(1+L)^{-1}$.

\begin{thm}\label{thm-main-bis}
Let $L$ be a kernel satisfying Condition \ref{con5}. Assume that $g$ is a bounded real function such that $| g(x) - 1| \le C |x|$ and there exists $\varepsilon \in (0,1)$ such that the subset $ \{ x \in \R^*: | g(x) - 1| > \varepsilon\}$ has a positive distance from $0$.
Then the following limit 
$$
\overline{S}_L[\log g^2](X) : =\lim_{\delta\to 0^{+}}    \sum_{x \in X, | x| \ge \delta}\log g(x)^2 - \E_{\mu_L} \left(\sum_{x \in X, | x| \ge \delta}\log g(x)^2 \right)
$$
exists for $\mu_L$-almost every configuration $X\in \Conf(\R^*)$.  Moreover, we have
$$
\exp (\overline{S}_L[\log g^2] ) \in L^1(\Conf(\R^*), \mu_L)
$$ 
and we have 
$$
 \mu_{gLg}(dX)=  \frac{\exp (\overline{S}_L[\log g^2] (X)  )}{\E_{\mu_L} \Big[ \exp (\overline{S}_L[\log g^2]) \Big]} \cdot \mu_L(dX).
$$
\end{thm}

\subsection{Remarks on Conditions \ref{con35}, \ref{con-four}, \ref{con5} and Proofs of Lemmas \ref{lem-V}, \ref{lem-Kg}, \ref{4-5}}

\begin{rem}\label{rem-K-block}
Let $L$ be an operator as in Condition \ref{con-four}, then $K_L$ has the following block form: 
\begin{align}\label{Kbf}
K_L = \left[   \begin{array}{cc}  VV^*( 1 + VV^*)^{-1} & V ( 1 + V^*V)^{-1}  \vspace{2mm}\\ - V^* ( 1 + VV^*)^{-1} & V^*V( 1 + V^*V)^{-1} \end{array} \right]. 
\end{align}
\end{rem}

\begin{rem}\label{rem-off-s}
By the block form \eqref{Kbf} of the operator $K_L$ and the following elementary order inequalities for positive operators 
$$
\frac{1}{1 + \| V\|^2}VV^* \le VV^*(1 + VV^*)^{-1} \le VV^*
$$
and
$$
 \frac{1}{1 + \| V\|^2}V^*V \le V^*V(1 + V^*V)^{-1} \le V^*V,
$$
we may formulate Condition \ref{con5} in terms of the kernel $K_L$ as follows:
for any $\varepsilon > 0$, we have
\begin{align}\label{off-s}
 \int_{| x| \ge \varepsilon} K_L(x,x) dx < \infty \an
\int_{| x | \le \varepsilon} x^2 K_L(x,x) dx   < \infty.
\end{align}
In particular, by \cite[Thm. 4]{DPP-S}, the first inequality in \eqref{off-s} implies that the relation
\begin{align}\label{fm-particle}
\#_{\R\setminus(-\varepsilon, \varepsilon)} (X) < \infty
\end{align}
 holds for $\PP_{K_L}$-almost every configuration $X\in \Conf(\R^*)$.
\end{rem}

\begin{rem}
Finite rank perturbation of the kernel $K_L$ will not affect the inequalities \eqref{off-s}.
\end{rem}

\begin{rem}
In Condition \ref{con35}, we require the boundedness of $L$ and hence of $V$. In general,  an operator $V$ given by a kernel  
$$
V(x,y) = \frac{  A^{+}(x) A^{-}(y)}{x-y}
$$
 such that $A\in L^2(\R^*, dx)$, is not necessarily bounded. See e.g., Propositions 2.2 and 2.3 in \cite{PV}.
\end{rem}

\begin{rem}
The operator norm of the operator $K_L$ given in \eqref{Kbf} satisfies
\begin{align}\label{K-contractive}
\| K_L \| < 1. 
\end{align}
Indeed, by Proposition \ref{prop-proj}, the operator $K_L$ is $J$-self-adjoint and $\widehat{K}$ is an orthogonal projection, by \cite[Prop. 7]{Lytvynov-J}, we have $\| K_L\| \le 1$. We shall exclude the possibility $\| K_L \| =1$. Indeed, if this were the case (i.e., $\| K_L \| =1$), then by \cite[Prop. 8]{Lytvynov-J},  we would get
\begin{align}\label{absurd}
\left\|\left[   \begin{array}{cc}  VV^*( 1 + VV^*)^{-1} & 0 \vspace{2mm}\\ 0 & V^*V( 1 + V^*V)^{-1} \end{array} \right] \right\| = 1. 
\end{align}
However, by functional calculus,  we have 
$$
\|VV^*( 1 + VV^*)^{-1}\| = \|  V^*V( 1 + V^*V)^{-1}\| = \frac{\| V\|^2}{1 + \|V\|^2} < 1.
$$
Hence \eqref{absurd} is not valid. This completes the proof of \eqref{K-contractive}.
\end{rem}

We will often use
\begin{rem}\label{rem-ab}
Let $a, b$ be two bounded linear operator on a Hilbert space. If $1+ab$ is invertible, then so is $1+ba$. We have
\begin{align}\label{ab-form}
a(1 + ba)^{-1} = (1 + ab)^{-1}a.
\end{align}
\end{rem}

\begin{proof}[Proof of Lemma \ref{4-5}]
Let $L$ be a kernel satisfying Condition \ref{con35}.  Let $\varepsilon > 0$. The simple estimate:
\begin{align*}
& \max \Big( \int_\varepsilon^\infty dx  \int_{-\infty}^0 \left|  \frac{  A^{+}(x) A^{-}(y)}{x-y} \right|^2 dy,  \int_0^\infty dx  \int_{-\infty}^{-\varepsilon} \left|  \frac{  A^{+}(x) A^{-}(y)}{x-y} \right|^2 dy \Big)
\\
& \le \frac{\|A^{+}\|_2^2 \|A^{-}\|_2^2}{\varepsilon^2}
\end{align*}
 shows that the operators with the kernels
$$
\chi_{(\varepsilon, \infty)}(x)  \frac{  A^{+}(x) A^{-}(y)}{x-y} \an  \frac{  A^{+}(x) A^{-}(y)}{x-y}\chi_{(-\infty, -\varepsilon)}(y) 
$$
are Hilbert-Schmidt. 

The  inequality: 
\begin{align*}
& \max\left(\int_0^\varepsilon dx\int_{-\infty}^0 x^2 \left|\frac{A^{+}(x)A^{-}(y)}{x-y}\right|^2dy,  \int_0^\infty dx\int_{-\varepsilon}^0 y^2 \left|\frac{A^{+}(x)A^{-}(y)}{x-y}\right|^2dy \right)
\\
& \le \|A^{+}\|_2^2  \|A^{-}\|_2^2
\end{align*}
implies that $\chi_{(0, \varepsilon)} x V$ and  $V y \chi_{(-\varepsilon, 0)}$ are also Hilbert-Schmidt.  

The Lemma \ref{4-5} is proved completely. 
\end{proof}

\begin{proof}[Proof of Lemma \ref{lem-V}]
By Proposition \ref{prop-proj}, the operator $\widehat{K}_L$ is self-adjoint, hence $K_L$ is $J$-self-adjoint. By Remark \ref{rem-K-block}, the operators $P_{+}K_LP_{+}$ and $P_{-}K_LP_{-}$ are non-negative.  

Let $\Delta_1, \Delta_2$ be compact subsets of $\R^*$ such that $\Delta_1 \subset \R_{+}$ and $\Delta_2\subset \R_{-}$. We now check that the operators $\chi_{\Delta_i} K_L \chi_{\Delta_i} ( i = 1, 2)$ are in trace-class. Let us verify this for  $i =1$. 
Since $VV^*$ is positive, we have
$$
VV^*( 1 + VV^*)^{-1}  \le VV^*,
$$
it follows that
 \begin{align*}
 0 \le \chi_{\Delta_1} K_L \chi_{\Delta_1}  =   \chi_{\Delta_1} VV^*( 1 + VV^*)^{-1}  \chi_{\Delta_1} \le  \chi_{\Delta_1} V  V^*  \chi_{\Delta_1}.
 \end{align*}
The assumption that $\chi_{\Delta_1}V$ is a Hilbert-Schmidt operator now implies that $\chi_{\Delta_1} K_L\chi_{\Delta_1} $ is a trace-class operator. The case  of $i =2$ is similar.

Finally, let us verify that $\chi_{\Delta_2}K_L\chi_{\Delta_1}$ is a Hilbert-Schmidt operator. Indeed, we have 
$$
\chi_{\Delta_2}K_L\chi_{\Delta_1}  = - \chi_{\Delta_2}   V^* ( 1 + VV^*)^{-1} \chi_{\Delta_1}  = - \chi_{\Delta_2}    ( 1 + V^*V)^{-1} V^*\chi_{\Delta_1}.
$$
Since the space of Hilbert-Schmidt operators is an ideal of the algebra $\mathscr{L}(L^2(\R))$ of all bouned linear operators, the assumption that $V^*\chi_{\Delta_1} = (\chi_{\Delta_1} V)^*$ is a Hilbert-Schmidt operator implies that $\chi_{\Delta_2}K_L\chi_{\Delta_1}$ is a Hilbert-Schmidt operator. The proof is complete.
\end{proof}

\begin{proof}[Proof of Lemma \ref{lem-Kg}]
We have 
$$
gLg = \left[  \begin{array}{cc} 0 & g \chi_{\R_{+}} V   g \chi_{\R_{-}}  \\  - g\chi_{\R_{-}}V^* g \chi_{\R_{+}}  & 0 \end{array}\right].
$$
Since $g$ is bounded, the operator $g \chi_{\R_{+}} V   g \chi_{\R_{-}} $ satisfies all the conditions in Condition \ref{con-four} imposed on the operator $V$.

Set $a = g$ and $b = gK_L(1-K_L)^{-1}$.  The
$$ 
1 + gLg = 1 + gK_L(1-K_L)^{-1}g = 1 + ba
$$
is invertible. By Remark \ref{rem-ab}, the operator $1 + ab = 1 + g^2 K_L(1 - K_L)^{-1}$ is also invertible. The identity \eqref{ab-form} now yields
\begin{align*}
K_{gLg}  & = gK_L(1-K_L)^{-1}  g ( 1 +  gK_L(1-K_L)^{-1}  g )^{-1} 
\\
& = gK_L(1-K_L)^{-1}   ( 1 +  g^2K_L(1-K_L)^{-1}  )^{-1} g
\\
& = g K_L ( 1 - K_L + g^2K_L)^{-1}g,
\end{align*}
which is the desired identity. The proof is complete.
\end{proof}

\subsection{A preliminary version of Theorem \ref{thm-main-bis}}
As usual,  given a function $h$ defined on $\R^*$, we define the multiplicative functional $\Psi[h]: \Conf(\R^{*}) \rightarrow \R$ by the following formula
\begin{align}\label{Mul-form}
\Psi[h](X)  = \prod_{x\in X} h(x), \, X \in \Conf(\R^*),
\end{align}
provided the right-hand side converges absolutely.

\begin{prop}\label{prop-pre}
Let $L$ be an operator satisfying Condition \ref{con5}. If $g$ is a bounded real function defined on $\R^*$ such that $ \supp ( g^2 -1)$ has a positive distance from the origin. 
Then
$$
 \mu_{gLg}  =  \frac{  \Psi[g^2] }{\E_{\mu_L} ( \Psi[g^2] )} \cdot    \mu_L.
$$
\end{prop}

\begin{proof}
By general theory on determinantal measures, it suffices to prove that for any continuous real function $h$ such that  $\supp (h-1)$ has a positive distance from the origin, we have 
\begin{align}\label{product}
\E_{\mu_{gLg}} \left( \Psi[h]\right)  =  \frac{\E_{\mu_L}\left( \Psi[h] \Psi[g^2] \right)}{  \E_{\mu_L}\left(\Psi[g^2] \right) } =  \frac{\E_{\mu_L}\left( \Psi[h g^2] \right) }{  \E_{\mu_L}\left(\Psi[g^2] \right) }.
\end{align}
Fix such  a  function $h$, set $\Delta = \supp(h-1) \cup \supp(g^2-1).$
Then $\Delta$ also has a positive distance from the origin. Obviously, $\supp(hg^2-1)\subset \Delta.$
By definition of determinantal point process and Theorem \ref{thm-01}, we have 
\begin{align*}
\E_{\mu_{gLg}} (\Psi[h]) &= \det(1+ (h-1)\chi_{\Delta} \cdot K_{gLg}  \cdot \chi_{\Delta} );
\\
\E_{\mu_L} (\Psi[ h g^2]) &= \det(1+ (h g^2-1)\chi_{\Delta}\cdot  K_L \cdot \chi_{\Delta} );
\\
\E_{\mu_L} (\Psi[g^2])&= \det(1+ (g^2-1)\chi_{\Delta}\cdot K_L \cdot \chi_{\Delta} ).
\end{align*}
Recall that by by Lemma \ref{lem-Kg}, the operator $K(g)$ satisfies all the conditions of Theorem \ref{thm-0} and is given by $K_{gLg}   = gK_L(1+(g^2-1)K_L)^{-1} g$. By \eqref{K-Delta}, we have 
$$
K_L^\Delta  = \chi_{\Delta}\cdot  K_L \cdot \chi_{\Delta} \in \mathscr{L}_{1|2}(L^2(\R)) \subset \mathscr{L}_{2}(L^2(\R)). 
$$
By Remark \eqref{rem-ab}, we get
\begin{align}\label{Kg-comp}
\begin{split}
\chi_{\Delta} \cdot  K_{gLg} \cdot \chi_{\Delta} & =  g  \chi_\Delta K_L ( 1 + (g^2 -1)  K_L)^{-1}  \chi_\Delta g  
\\
 & = g  \chi_\Delta K_L \Big[( 1 + \chi_\Delta (g^2 -1)\chi_\Delta  K_L)^{-1}  \chi_\Delta\Big]g
 \\
 & =  g \chi_\Delta K_L \Big[  \chi_\Delta  ( 1 + (g^2-1) \chi_\Delta K_L \chi_\Delta)^{-1}  \Big]g  
 \\
 & = g  K_L^{\Delta} ( 1 + (g^2 -1) K_L^{\Delta})^{-1}g. 
 \end{split}
 \end{align}
Observe that we can write
 \begin{align*}
& K_L^{\Delta} ( 1 + (g^2 -1) K_L^{\Delta})^{-1} 
\\
 = & \Big[  K_L^{\Delta} ( 1 + (g^2- 1) K_L^{\Delta}) + K^\Delta (1-g^2)K_L^\Delta
\Big]( 1 + (g^2 -1) K_L^{\Delta})^{-1} 
\\
 =&  K_L^\Delta +  K_L^\Delta (1-g^2)K_L^\Delta (1 + (g^2-1)K_L^\Delta)^{-1},
\end{align*}
then by H\"older inequality, we have 
$$
\| K_L^\Delta (1-g^2)K_L^\Delta\|_1 \le \| K_L^\Delta\|_2 \| (1-g^2)K_L^\Delta\|_2 \le \| g^2-1\|_\infty \| K_L^\Delta\|_2^2< \infty,
$$
that is, the operator $K_L^\Delta (1-g^2)K_L^\Delta (1 + (g^2-1)K_L^\Delta)^{-1}$ is in trace-class. It follows that 
 \begin{align*}
  K_L^{\Delta} ( 1 + (g^2 -1) K_L^{\Delta})^{-1}  \in \mathscr{L}_{1|2} (L^2(\R)).
 \end{align*}
Thus we have\begin{align*}
\det(1 + (h-1)\chi_{\Delta} K_{gLg} \chi_{\Delta}) & = \det   ( 1 +  (h-1) g K_L^{\Delta} ( 1 + (g^2 -1) K_L^{\Delta})^{-1}g)
\\
 &  =  \det   ( 1 +  (h-1) g^2 K_L^{\Delta} ( 1 + (g^2 -1) K_L^{\Delta})^{-1}).
\end{align*}
An application of the identity
\begin{align*}
1 +  (h-1) g^2 K_L^{\Delta} ( 1 + (g^2 -1) K_L^{\Delta})^{-1} = \Big[1 + (hg^2 - 1) K_L^\Delta \Big] ( 1 + (g^2 -1)K_L^{\Delta})^{-1}
\end{align*}
yields that
\begin{align*}
\det(1 + (h-1)\chi_{\Delta} K_{gLg} \chi_{\Delta})   = \frac{\det(1 + (hg^2 - 1) K_L^\Delta )}{\det(1 + (g^2 -1) K_L^{\Delta})}.
\end{align*}
This shows the desired identity \eqref{product}. The proof of Proposition \ref{prop-pre} is complete.
 \end{proof}

The following lemma will be useful for us.

\begin{lem}\label{lem-conjugation}
Assume that $L$ is an operator satisfying Condition \ref{con-four}. 
\begin{itemize}
\item Let $\alpha: \R^*\rightarrow \C$ be any measurable function with non-zero constant modulus. Then $\alpha L \alpha^{-1}$ induces a determinantal measure $\mu_{\alpha L \alpha^{-1}} =  \PP_{K_{{\alpha L \alpha^{-1}}}}$ and  $\mu_{\alpha L \alpha^{-1}} = \mu_L$.
 
\item Let $\lambda \ne 0$ be a numerical constant and let $c_\lambda$ be the function defined by 
\begin{align}\label{c-form}
c_\lambda =  \lambda \chi_{\R_{+}}+  \lambda^{-1} \chi_{\R_{-}}.
\end{align}
Then $c_\lambda L c_\lambda = L$.
\end{itemize}
\end{lem}

\begin{proof}
By assumption $\alpha =  a  \gamma$ with $a = | \alpha|>0$ a numerical constant and $\gamma$ a measurable function with values in the unit circle. We shall see that $\alpha L \alpha^{-1}$ satisfies also Condition \ref{con-four}. Indeed, 
$$
\alpha L \alpha^{-1} = \gamma L \bar{\gamma} = \left[  \begin{array}{cc} 0 &   \gamma^{+} V\bar{\gamma}^{-} \\  - \gamma^{-} V^* \bar{\gamma}^{+}  & 0 \end{array}\right] =  \left[  \begin{array}{cc} 0 &   \gamma^{+} V\bar{\gamma}^{-} \\  - ( \gamma^{+} V\bar{\gamma}^{-})^*  & 0 \end{array}\right].
$$
Hence the determinantal measure $\mu_{\alpha L \alpha^{-1}}$ is well-defined. The coincidence of $\mu_{\alpha L \alpha^{-1}}$ and $\mu_L$ is an easy consequence of the fact that $K_{\alpha L \alpha^{-1}} = \alpha K_L \alpha^{-1}$ and 
$$
\det(\alpha(x_i) K_L(x_i, x_j) \alpha(x_j)^{-1})_{1\le i, j \le n} = \det( K_L(x_i, x_j) )_{1\le i, j \le n}.
$$

The second assertion is an easy consequence of the following identity
$$
\left[  \begin{array}{cc} \lambda &   0 \\  0 & \lambda^{-1}\end{array}\right]  \left[  \begin{array}{cc} 0 &   V \\  - V^*  & 0 \end{array}\right]  \left[  \begin{array}{cc} \lambda &   0 \\  0 & \lambda^{-1}\end{array}\right]  = \left[  \begin{array}{cc} 0 &   V \\  - V^*  & 0 \end{array}\right].
$$
\end{proof}

\subsection{Regularization of additive and multiplicative functionals}\label{sec-regularization}
\subsubsection{Additive functionals}
Assume now that $L$ is a kernel satisfying Condition \ref{con5}. Recall that we set 
$$
\mu_L = \PP_{K_L}.
$$ 
Let $f: \R^*\rightarrow \C$ be a Borel function. Then we write 
\begin{align}\label{defn-twist-statistics}
T[f] (X)  = S[f^\circ] (X) = \sum_{x \in X} \sgn(x) f(x),
\end{align}
provided the right hand side converges absolutely, otherwise, $T[f]$ is not defined at the configuration $X$.   

If $T[f]$ is $\mu_L$-almost surely defined and $ T[f] \in L_1(\Conf(\R^*), \mu_L)$, then we set 
\begin{align}\label{n-t-a}
\overline{T}_L[f]: = T[f] - \E_{\mu_L} (T[f]).
\end{align}
Following  the idea in \cite{BQS}, we will now provide a sufficient condition such that $\overline{T}_L[f]$ can be defined even when $T[f]$ is not. Set
\begin{align}
\mathscr{V}_L(f) = \frac{1}{2} \iint_{\R^2} | f(x) - f(y) |^2 |K_L(x,y)|^2 dxdy. 
\end{align}
Note that for any $\lambda\in \C$,  we have $\mathscr{V}_L(f+ \lambda) = \mathscr{V}_L(f).$ Note also that 
\begin{align}\label{V-small}
\mathscr{V}_L(f) \le  \int_\R |f(x)|^2 K_L(x,x)dx.
\end{align}
By Lemma \ref{lem-var}, if $T[f]\in L_2(\Conf(\R^*), \mu_L)$, then $\mathscr{V}_L(f)< \infty$ and 
\begin{align}\label{sigma-2-norm}
\Var_{\mu_L} (T[f])   = \E_{\mu_L} | \overline{T}_L[f]|^2 =  \mathscr{V}_L(f).
\end{align}

\begin{defn}
Let $\mathscr{N}_0(L)$ be the linear space of Borel functions $f: \R^* \rightarrow \C$ such that there exist $\varepsilon>0$, depending on $f$, so that
$$
\supp(f) \subset \{x\in \R: |x|\ge \varepsilon\} \an \int_\R |f(x)|^2 K_L(x,x)dx < \infty.
$$
\end{defn}

\begin{defn}
Let $\mathscr{N}(L)$ be the linear space of Borel functions $f: \R^* \rightarrow \C$ such that 
\begin{align}
\mathscr{V}_L(f) = \frac{1}{2} \iint_{\R^2} | f(x) - f(y) |^2 |K_L(x,y)|^2 dxdy < \infty; \label{v-finite}
\\
\label{summable}
\int_{|x|\ge \varepsilon} |f(x)|^2 K_L(x,x)dx < \infty, \text{ for all $\varepsilon >0;$}
\\
\label{zero-at-s}
\lim_{\varepsilon\to 0^+} \iint_{|x|\le \varepsilon, |y|\ge \varepsilon} |f(x)|^2 |K_L(x,y)|^2dxdy =0.
\end{align}
We endow the linear space $\mathscr{N}(L)$ with a Hilbert space structure $d_{\mathscr{N}(L)}$ by the formula
\begin{align*}
d_{\mathscr{N}(L)}(f, g)   = \|f - g\|_{\mathscr{N}_L}  : =\sqrt{\mathscr{V}_L(f-g)}.
\end{align*}
\end{defn}

\begin{rem}
If $f\in \mathscr{N}_0(L)$, then by Cauchy-Buniakovsky-Schwarz inequality and the first inequality in \eqref{off-s}, we have 
$$
\int_\R |f(x)|K_L(x,x)dx \le    \left(\int_{\R} | f(x)|^2 K_L(x,x)dx \cdot \int_{\supp(f)} K_L(x,x)dx\right)^{1/2}  <\infty.
$$ 
This means that $T[f] \in L_1(\Conf(\R^*), \mu_L)$, hence $\overline{T}_L[f]$ is well-defined by formula \eqref{n-t-a}. Moreover,  by the relations \eqref{V-small} and \eqref{sigma-2-norm},  we actually have 
$$
\overline{T}_L[f] \in L_2(\Conf(\R^*), \mu_L).
$$
\end{rem}

\begin{prop}
We have the inclusion $$\mathscr{N}_0(L) \subset \mathscr{N}(L).$$ Moreover, $\mathscr{N}_0(L)$ is dense in $\mathscr{N}(L)$.  More precisely, if $f$ is a function in $\mathscr{N}(L)$,  then for any $\varepsilon > 0$, the truncated function  $f \chi_{\R\setminus(-\varepsilon, \varepsilon)}$ is in $\mathscr{N}_0(L)$ and we have
\begin{align}\label{cut-conv}
\lim_{\varepsilon \to 0^+} \mathscr{V}_L(  f \chi_{\R\setminus(-\varepsilon, \varepsilon)}  - f)= 0.
\end{align}
\end{prop}

\begin{proof}
The inclusion $\mathscr{N}_0(L) \subset \mathscr{N}(L)$ follows from their definitions and the following inequality
 $$
\iint_{|x|\le \varepsilon, |y|\ge \varepsilon} |f(x)|^2 |K_L(x,y)|^2dxdy \le \int_{|x|\le \varepsilon} |f(x)|^2 K_L(x,x) dx,
$$
By definition of $\mathscr{N}(L)$, we have $f \chi_{\R\setminus(-\varepsilon, \varepsilon)} \in \mathscr{N}_0(L)$. Since
\begin{align*}
&\mathscr{V}_L(  f \chi_{\R\setminus(-\varepsilon, \varepsilon)} - f)  =  \frac{1}{2} \iint_{\R^2} |f \chi_{[-\varepsilon, \varepsilon]}(x)  - f\chi_{[-\varepsilon, \varepsilon]}(y)|^2 |K_L(x,y)|^2 dxdy 
\\
& \le  \iint_{|x|\le \varepsilon, |y|\le \varepsilon} |f(x) -  f(y)|^2|K_L(x,y)|^2dxdy  +  \iint_{|x|\le \varepsilon, |y|\ge \varepsilon} |f(x)|^2 |K_L(x,y)|^2dxdy
\\
& + \iint_{|x|\ge \varepsilon, |y|\le \varepsilon} |f(y)|^2 |K_L(x,y)|^2dxdy
\\ 
& = \iint_{|x|\le \varepsilon, |y|\le \varepsilon} |f(x) -  f(y)|^2|K_L(x,y)|^2dxdy  +  2 \iint_{|x|\le \varepsilon, |y|\ge \varepsilon} |f(x)|^2 |K_L(x,y)|^2dxdy.
\end{align*}
By the assumption $\mathscr{V}_L(f) < \infty$ and the relation \eqref{zero-at-s},  we get the desired relation \eqref{cut-conv}.
\end{proof}

\begin{prop}\label{prop-em}
The isometric embedding 
$$
\begin{array}{ccc}
\overline{T}: \mathscr{N}_0(L)& \hookrightarrow & L_2(\Conf(\R^*), \mu_L)
\\
f & \mapsto & \overline{T}_L[f]
\end{array}
$$
 extends uniquely to an isometric embedding $\mathscr{N}(L)\hookrightarrow L_2(\Conf(\R^*), \mu_L)$. 
\end{prop}

\begin{defn}\label{defn-twist-linear}
Given a function $f \in \mathscr{N}(L)$, by slightly abusing the notation, we denote by $\overline{T}_L[f]$ the image of $f$ under the embedding map $\mathscr{N}(L)\hookrightarrow L_2(\Conf(\R^*), \mu_L)$ in Proposition \ref{prop-em}. We will call $\overline{T}_L[f]$ the normalized twisted additive functional corresponding to $f$ and $\mu_L$. 
\end{defn}

\begin{rem}\label{rem-null}
For all $f \in \mathscr{N}(L)$, we have $\E_{\mu_L}( \overline{T}_L[f] )=0$.
\end{rem}

\begin{rem}\label{pt-limit}
If $f\in \mathscr{N}(L)$, then up to passing to a sequence $\varepsilon_n$ tending to zero if necessary, we may write the following {\it pointwise} relation: for $\mu_L$-almost every configuration $X\in \Conf(\R^*)$, 
\begin{align}\label{cut-convergence}
\overline{T}_L[f] (X)= \lim_{\varepsilon \to 0^+}    \left( \sum_{x \in X, | x | \ge \varepsilon} \sgn(x) f(x)  - \E_{\mu_L}\sum_{x \in X, | x | \ge \varepsilon} \sgn(x) f(x) \right).
\end{align}
\end{rem}

\subsubsection{Multiplicative functionals}

\begin{notation}
Let  $f:\R^*\rightarrow \C$ be a measurable function, denote 
$$
f^{+} = f \chi_{\R_{+}} \an f^{-} = f \chi_{\R_{-}}.
$$  
If the essential support $\supp (f^{-})$ of the function $f^{-}$ is the whole negative semi-axis $\R_{-}$, then we may define
\begin{align}\label{def-f-vee}
f^\vee(x) := f^{+}(x) + (f^{-}(x))^{-1}. 
\end{align}
\end{notation}

\begin{defn}\label{defn-multiplicative}
Given a function $g: \R^* \rightarrow [0, \infty]$ such that  $\{x\in \R^*: g(x) =0\}$ is Lebesgue negligible  and $\log g \in \mathscr{N}(L)$, then we set 
\begin{align*}
 \widetilde{\Psi}_{L}[g] = \exp(\overline{T}_L[\log (g^\vee)]),
\end{align*}
where by definition \eqref{def-f-vee}, $g^\vee(x) := g(x) \chi_{\R_{+}}(x) + g(x)^{-1} \chi_{\R_{-}}(x)$.
If moreover, $\E_{\mu_L} \widetilde{\Psi}_L[g]$ is finite,  then we define
\begin{align*}
 \overline{\Psi}_L[g] = \frac{\widetilde{\Psi}_L[g]}{\E_{\mu_L} \widetilde{\Psi}_L[g]}.
\end{align*}
\end{defn}

\begin{rem}
If $g$ is a function such that $\log g \in \mathscr{N}(L)$, then $ \E_{\mu_L} \widetilde{\Psi}_L[g] \in [1, \infty]$. Indeed, by Jensen's inequality and Remark \ref{rem-null}, we have 
\begin{align}\label{jensen}
\E_{\mu_L} \widetilde{\Psi}_L[g] =  \E_{\mu_L}  \exp(\overline{T}_L[\log (g^\vee)])  \ge \exp(\E_{\mu_L} ( \overline{T}_L{\log g^\vee})) =1.
\end{align}
\end{rem}

The formalism of regularized multiplicative functional  $\overline{\Psi}_L[g]$ now allows us state the following
\begin{thm}\label{thm-sub-main}
Let $g: \R^* \rightarrow [0, \infty)$ be a non-negative bounded function. Assume that there exists $\varepsilon \in (0, 1) $  such that $E_\varepsilon = \{ x \in \R^*: | g(x)^2 - 1| > \varepsilon\}$ has a positive distance from the origin and 
\begin{align}\label{reg-condition}
 \int_{\R} | g(x) -1|^2 K_L(x,x)dx <\infty.
\end{align}
Then $\log g \in \mathscr{N}(L)$ and $\widetilde{\Psi}_L[g^2]\in L_1(\Conf(\R^*), \mu_L)$. Moreover, we have 
\begin{align}\label{boss-rel}
\mu_{g Lg } = \overline{\Psi}_L[g^2] \cdot \mu_L . 
\end{align}
\end{thm}

\subsection{Proof of Theorem \ref{thm-sub-main}}

\begin{defn}
Let $\mathscr{M}_2(L)$ denote the set of functions $g$ on $\R$ such that 
\begin{itemize}
\item[(1)] $0 < \inf_\R g \le \sup_\R g < \infty$;
\item[(2)] $\int_\R |g(x) - 1|^2 K_L(x,x)dx < \infty$.
\end{itemize}
\end{defn}

Recall that by definition \eqref{def-f-vee},  to a function $g$, we asign $g^\vee$ in the following way: 
$$
g^\vee(x) := g(x) \chi_{\R_{+}}(x) + g(x)^{-1} \chi_{\R_{-}}(x).
$$

\begin{prop}\label{prop-rn}
Let $g \in \mathscr{M}_2(L)$. Then $\log g$ and $\log (g^\vee)$ are functions in $\mathscr{N}(L)$. In particular, the functional $\widetilde{\Psi}_L[g]=  \exp(\overline{T}_L[\log (g^\vee)])$ is well-defined.
Moreover, we have 
\begin{align}\label{RN}
\mu_{gLg} = \overline{\Psi}_L[g^2]  \cdot \mu_L.
\end{align}
\end{prop}
We postpone its proof to the next section.

Now we are in a position to prove Theorem \ref{thm-sub-main}. But first, let us note that for a function $g$ as in Proposition \ref{prop-pre}, the regularized multiplicative functional $\overline{\Psi}_L[g^2]$ defined as above is also expressed by $\overline{\Psi}_L[g^2] = C \Psi[g^2]$ for a certain constant $C>0$.

\begin{proof}[Proof of Theorem \ref{thm-sub-main}]
Let $\varepsilon \in (0,1)$ be such that $\{ x \in \R^*: | g(x) - 1| > \varepsilon\}$ has a positive distance from the origin. Set $g_1, g_2$ to be two positive functions  determined by
\begin{align}\label{g1}
g_1 = (g - 1)\chi_{\{x\in \R^*: |g(x)-1| \le \varepsilon\}} + 1.
\end{align}
\begin{align}\label{g2}
g_2 = (g - 1)\chi_{\{x\in \R^*: |g(x)-1| > \varepsilon\}} + 1.
\end{align}
By definition, $g = g_1 g_2$. Note that $1- \varepsilon \le  \inf_\R g_1 \le \sup_\R g_1 \le 1 + \varepsilon$. This combining with assumption \eqref{reg-condition}  shows that the function $g_1$ is in $\mathscr{M}_2(L)$. Hence by Proposition \ref{prop-rn}, we have 
\begin{align}\label{g1-RN}
\mu_{g_1 L g_1} = \overline{\Psi}_L[g_1^2]\cdot \mu_L.
\end{align}
Now since $\supp(g_2 -1)$ has a positive distance from the origin and $gL g = g_2(g_1L g_1)g_2$,  by Proposition \ref{prop-pre}, we have
\begin{align}\label{g2-RN}
\mu_{g L g}  = \frac{\Psi[g_2^2]}{\E_{\mu_{g_1 L g_1}} \Psi[g_2^2] } \cdot \mu_{g_1 L g_1} .
\end{align}
Combining \eqref{g1-RN} and \eqref{g2-RN}, we get 
\begin{align}
\mu_{gLg} =  \frac{\Psi[g_2^2]}{\E_{\mu_{g_1 L g_1}} \Psi[g_2^2] } \cdot \overline{\Psi}_L[g_1^2]\cdot \mu_L.
\end{align}
Since 
$$ 
\frac{\Psi[g_2^2]}{\E_{\mu_{g_1 L g_1}} \Psi[g_2^2] } \cdot \overline{\Psi}_L[g_1^2]  = C_1 \widetilde{\Psi}_L[g_2^2] \cdot C_2 \widetilde{\Psi}_L[g_1^2]  = C_1C_2 \widetilde{\Psi}_L[g_1^2g_2^2]  = C_1C_2 \widetilde{\Psi}_L[g^2],
$$ 
and 
$$
\int  C_1C_2 \widetilde{\Psi}_L[g^2] d\mu_L = 1,
$$
we get 
$$
\frac{\Psi[g_2^2]}{\E_{\mu_{g_1 L g_1}} \Psi[g_2^2] } \cdot \overline{\Psi}_L[g_1^2]  = \overline{\Psi}_L[g^2],
$$
hence we complete the proof of the desired relation \eqref{boss-rel}.
\end{proof}

\subsection{Proof of Proposition \ref{prop-rn}}
Let us endow $\mathscr{M}_2(L)$ with a metric $d_{\mathscr{M}_2(L)}$ by setting
$$
d_{\mathscr{M}_2(L)} (g_1, g_2) = \sqrt{\int_\R | g_1(x) - g_2(x) |^2 K_L(x,x ) dx} .
$$ 
By definition, $\mathscr{M}_2(L)$ is a semigroup under pointwise multiplication. Clearly, if $g$ is a function in $\mathscr{M}_2(L)$, then so is $g^\vee$.

We shall first prove the following
\begin{lem}\label{good-log}
Let $g \in \mathscr{M}_2(L)$. Then $\log g$ and $\log (g^\vee)$ are functions in $\mathscr{N}(L)$. 
\end{lem}

\begin{proof}[Proof of Lemma \ref{good-log}]
Assume that $g \in \mathscr{M}_2(L)$. Then there exist $c, C >0$ such that $c\le g(x) \le C$.
The boundedness of the function  $\log g$ combining with the assumption \eqref{off-s} yields the inequality \eqref{summable} for $\log g$. 

Now since the function $| \log t - (t -1)| /(t-1)^2$ is bounded on the interval $[c, C]$, there exists $C'>0$ such that 
\begin{align}\label{D-log-other}
| \log g (x ) - (g (x)  - 1)| \le C' (g(x) - 1)^2.
\end{align}
By taking $C'' = 1 + C' \max (|C-1|, |c-1|)$, we have 
\begin{align}\label{log-and-else}
|\log g(x) | \le C''|g(x)-1|.
\end{align}
It follows that 
\begin{align*}
\int_{\R} | \log g(x) |^2 K_L(x,x) dx \le (C'')^2 \int_\R |g(x) - 1|^2 K_L(x,x)dx < \infty.
\end{align*}
Hence by applying \eqref{V-small}, we have $\mathscr{V}_L(\log g) < \infty$.  Following from \eqref{log-and-else}, we also have 
\begin{align*}
& \limsup\limits_{\varepsilon\to 0^+} \iint_{|x|\le \varepsilon, |y|\ge \varepsilon} | \log g (x) |^2 |K_L(x,y)|^2dxdy
\\
 \le&
  \lim\limits_{\varepsilon\to 0^+} (C'')^2\iint_{|x|\le \varepsilon, |y|\ge \varepsilon} | g(x) -1|^2 |K_L(x,y)|^2dxdy 
 \\
 \le & 
 \lim\limits_{\varepsilon\to 0^+} (C'')^2\int_{|x|\le \varepsilon} | g(x) -1|^2 K_L(x,x)dx  =0.
\end{align*}
This completes the proof that $\log g \in \mathscr{N}(L)$. The same argument for $\log (g^\vee)$ since $g\in \mathscr{M}_2(L)$ implies that $g^\vee \in \mathscr{M}_2(L)$.
\end{proof}

\begin{prop}\label{prop-con}
If  $g\in \mathscr{M}_2(L)$, then $\widetilde{\Psi}_L[g] \in L_1(\Conf(\R^*), \mu_L)$. Moreover, the mappings 
$$
 g \rightarrow \widetilde{\Psi}_L[g] \an g \rightarrow \overline{\Psi}_L[g]
$$
are both continuous from $\mathscr{M}_2(L)$ to $L_1(\Conf(\R^*), \mu_L)$.
\end{prop}

\begin{proof}[Proof of Proposition \ref{prop-rn}]
Let $E_{n}\subset \R^*$ be a sequence of compact subsets exhausting $\R^*$ and set 
$$
g_n = 1 + (g-1)\chi_{E_n}.
$$
Clearly,  we have $g_n^2 = 1 + (g^2-1)\chi_{E_n}$ and 
\begin{align}\label{gn-to-g}
g_n^2 \xrightarrow[d_{\mathscr{M}_2(L)}]{n\to \infty} g^2.
\end{align}

Claim: $K(g_n)$ converges to $K(g)$ in the space of locally $\mathscr{L}_{1|2}$-operators. Indeed, by the block forms of $K(g_n)$ and $K(g)$ as in \eqref{Kbf}, we need to show that for any compact subsets $\Delta_1, \Delta_2$ of $\R^*$ such that $\Delta_1 \subset \R_{+}$ and $\Delta_2 \subset \R_{-}$, we have 
\begin{align}\label{first}
\chi_{\Delta_1}g_n V V^* g_n (1 + g_n V V^* g_n)^{-1} \chi_{\Delta_1} \xrightarrow[\text{in trace class}]{n\to \infty} \chi_{\Delta_1}g V V^* g(1 + g V V^* g)^{-1} \chi_{\Delta_1};
\end{align}
\begin{align}\label{second}
\chi_{\Delta_2} V^*g_n^2 V (1 +  V^* g_n^2V )^{-1} \chi_{\Delta_2} \xrightarrow[\text{in trace class}]{n\to \infty} \chi_{\Delta_2} V^* g^2 V (1 +  V^* g^2 V )^{-1} \chi_{\Delta_2};
\end{align}
\begin{align}\label{third}
\chi_{\Delta_1}g_n  V  (1 +  V^*g_n^2 V)^{-1} \chi_{\Delta_2} \xrightarrow[\text{in trace class}]{n\to \infty} \chi_{\Delta_1}g  V (1 +  V^*g^2 V )^{-1} \chi_{\Delta_2}.
\end{align}
Let us prove the first relation \eqref{first}, the proof of second and third relations are similar to that of the first one.  First of all,
$$
g_nV\xrightarrow[s.o.t.]{n\to \infty} gV \an V^* g_n \xrightarrow[s.o.t.]{n\to \infty} V^*g,
$$
where s.o.t. stands for the strong operator topology. Hence 
we have 
$$
V^*g_n^2V\xrightarrow[s.o.t.]{n\to \infty} V^*g^2V 
$$
by continuity of the inverse mapping with respect to strong operator topology (cf. e.g. \cite[Lem. 3.2.]{Kadison-SOT}), we have 
$$
(1 + V^*g_n^2V )^{-1}      \xrightarrow[s.o.t.]{n\to \infty} (1 + V^*g^2V )^{-1}.
$$
  Note also that we have  
$$
 \chi_{\Delta_1}  g_nV  \xrightarrow[\text{Hilbert-Schmidt}]{n\to \infty}   \chi_{\Delta_1}  gV.
$$
Combining the above facts and \cite[Thm. 1]{Grumm}, we obtain that  
$$
\chi_{\Delta_1}g_n V  (1 + V^*g_n^2V )^{-1}   \xrightarrow[\text{Hilbert-Schmidt}]{n\to \infty} \chi_{\Delta_1}g V  (1 + V^*g^2V )^{-1}   .
$$ 
Now by using the following identity
$$
\chi_{\Delta_1}g_n V V^* g_n (1 + g_n V V^* g_n)^{-1} \chi_{\Delta_1} = \chi_{\Delta_1}g_n V  (1 +   V^* g_n^2V)^{-1} V^* g_n\chi_{\Delta_1}
$$
and the triangular inequalities, we conclude the proof of the desired relation \eqref{first}.

As a consequence of our claim, we have the weak convergence of the sequence of measures $\mu_{g_nL g_n}$ to the measure $\mu_{gLg}$. By Proposition \ref{prop-pre}, we also have 
$$
\mu_{g_nL g_n} = \overline{\Psi}_L[g_n^2]\cdot \mu_L.
$$
By Proposition \ref{prop-con} and \eqref{gn-to-g}, $\overline{\Psi}_L[g_n^2]$ converges to $\overline{\Psi}_L[g^2]$ in $L_1(\Conf(\R^*), \mu_L)$. As a consequence, we get the desired relation \eqref{RN}. 
\end{proof}

The rest of this section is devoted to the proof of Proposition \ref{prop-con}.

\begin{lem}\label{fKf-norm}
Let $f:\R^* \rightarrow \C$ be a Borel function such that
$
\int_{\R} | f(x)|^4 K_L(x,x)dx <\infty.
$
Then $fK_Lf$ is a Hilbert-Schmidt operator and 
$$
\| fK_Lf\|_2^2 \le \int_{\R} | f(x)|^4 K_L(x,x)dx.
$$
\end{lem}

\begin{proof}
We have 
\begin{align*}
\| fK_Lf\|_2^2 & = \iint\limits_{\R^2} |f(x)|^2 |f(y)|^2 |K_L(x,y)|^2 dxdy
\\
&\le \left( \iint\limits_{\R^2}  |f(x)|^4  |K_L(x,y)|^2 dxdy\right)^{1/2} \left( \iint\limits_{\R^2}  |f(y)|^4  |K_L(x,y)|^2 dxdy\right)^{1/2}
\\
 & = \left(  \int_{\R} | f(x)|^4 K_L(x,x)dx \right)^{1/2} \left(  \int_{\R} | f(y)|^4 K_L(y,y)dy\right)^{1/2}
 \\
 & =  \int_{\R} | f(x)|^4 K_L(x,x)dx.
\end{align*}
The proof is complete.
\end{proof}

\begin{rem}
By definition, if  $g_1$ and $g_2$ are two functions such that $\log g_1, \log g_2 \in \mathscr{N}(L)$, then  
\begin{align*}
 \widetilde{\Psi}_L[g_1g_2] = \widetilde{\Psi}_L[g_1]\widetilde{\Psi}_L[g_2].
\end{align*}
\end{rem}

\begin{lem}\label{lem-tilde-2-norm}
For any $\varepsilon>0, M>0$ so that $\varepsilon < 1 < M$, there exists a constant $C_{\varepsilon, M} >0$ such that if $g \in \mathscr{M}_2(L)$ satisfies  $\varepsilon \le \inf_\R g \le \sup_\R g \le M$, 
then 
\begin{align*}
\log \E_{\mu_L} (|\widetilde{\Psi}_L[g] |^2) \le C_{\varepsilon, M} \int_\R | g(x) -1|^2 K_L(x,x)dx.
\end{align*}
\end{lem}

\begin{proof}
By multiplicativity, it suffices to prove 
\begin{align}\label{no-square}
\log \E_{\mu_L}(\widetilde{\Psi}_L[g]) \le C_{\varepsilon, M} \int_\R | g(x) -1|^2 K_L(x,x)dx.
\end{align}
Since $g^\vee \in \mathscr{M}_2(L)$, by Lemma \ref{good-log},  $\log (g^\vee) \in \mathscr{N}(L)$, hence by Remark \ref{pt-limit}, passing to  a sequence $\delta_n$ if necessary, the functional $\overline{T}_{\log (g^\vee)}$ can be approximated pointwisely by 
$$
\overline{T}_{(\log (g^\vee))\chi_{\R\setminus (-\delta, \delta)}} = \overline{T}_{\log (g^\vee \chi_{\R\setminus (-\delta, \delta)} + \chi_{[-\delta, \delta]})}.
$$
Thus by Fatou's lemma, it suffices to establish \eqref{no-square} in the case when $\supp(g-1)$ is contained in some $\R \setminus(-\delta, \delta)$. In this case, the usual multiplicative functional $\Psi[g]$ is well-defined and we have 
$$
\widetilde{\Psi}_L[g] = \exp(S[\log g] - \E_{\mu_L} S[\log g]) = \frac{\Psi [g]}{\exp( \E_{\mu_L} S[\log g])}.
$$
Now by the very definition of determinantal point process $\mu_L = \PP_{K_L}$, we have
$$
\E_{\mu_L} S_{\log g}  = \int_{\R} \log g(x) K(x,x)dx
$$
and
$$
 \E_{\mu_L}\Psi[g] = \det(1  + \sqrt{g-1} K_L\sqrt{g-1}).
$$
By \cite[Thm. 6.4]{Simon-det}, if we denote $A = \sqrt{g-1} K_L\sqrt{g-1}$, we have 
\begin{align*}
|\det( 1 + A) \exp(-\tr(A))| \le \exp(\frac{1}{2}\|A\|_2^2).
\end{align*}
Hence by Lemma \ref{fKf-norm}, we have 
\begin{align*}
\log \E_{\mu_L}\Psi [g] & \le \tr (A)  + \frac{1}{2} \| A\|_2^2
\\
& \le \int_\R (g(x)-1) K_L(x,x) dx +\frac{1}{2} \int_\R | g(x) -1|^2 K_L(x,x)dx.
\end{align*}
An application of \eqref{D-log-other} to the function $g$ yields the existence of a constant $C_{\varepsilon, M}>0$, such that
\begin{align*}
\left| \int_\R  \log g (x)  K_L(x,x) dx - \int_\R (g(x) - 1) K_L(x,x) dx\right| \le  C_{\varepsilon, M}   \int_\R | g(x) -1|^2 K_L(x,x)dx.
\end{align*}
By setting $C_{\varepsilon, M}' = C_{\varepsilon, M}  + \frac{1}{2}$, we obtain 
\begin{align*}
\log \E_{\mu_L}\widetilde{\Psi}_L[g] = \log \E_{\mu_L} \Psi[g] - \E_{\mu_L} S[\log g] \le  C_{\varepsilon, M}' \int_\R | g(x) -1|^2 K_L(x,x)dx.
\end{align*}
The proof is complete.
\end{proof}

 \begin{lem}\label{lem-l1}
 Let $\varepsilon>0, M>0$ be two positive numbers such  that $\varepsilon < 1 < M$.  There exists a constant $C>0$ depending on $\varepsilon, M$,  such that if $g_1, g_2 \in \mathscr{M}_2(L)$ satisfy
\begin{align*}
\varepsilon \le \inf_\R g_1 \le \sup_\R g_1 \le M, \quad  \varepsilon \le \inf_\R g_2 \le \sup_\R g_2 \le M,
\end{align*}
then we have
\begin{align*}
\frac{(\E_{\mu_L}|\widetilde{\Psi}_L[g_1]  - \widetilde{\Psi}_L[g_2] |)^2 }{ \E_{\mu_L} (|\widetilde{\Psi}_L[g_1]|^2)}
\le &  \exp\Big( C \int_\R | g_1(x) -g_2(x)|^2 K_L(x,x)dx\Big)-1.
\end{align*}
 \end{lem}

\begin{proof}
Set $g = g_2/g_1$. Since $\widetilde{\Psi}_L[g_1]  - \widetilde{\Psi}_L[g_2] = \widetilde{\Psi}_L[g_1] (1  - \widetilde{\Psi}_L[g])  $, we have  
\begin{align}\label{cs}
(\E_{\mu_L}|\widetilde{\Psi}_L[g_1]  - \widetilde{\Psi}_L[g_2] |)^2  \le  \E_{\mu_L} (|\widetilde{\Psi}_L[g_1]|^2)    \cdot \E_{\mu_L}(  |\widetilde{\Psi}_L[g] -1 |^2).
\end{align}
By the inequality \eqref{jensen}, we have 
\begin{align}\label{a-jensen}
\E_{\mu_L} ( |\widetilde{\Psi}_L[g] -1 |^2) = \E_{\mu_L} (|\widetilde{\Psi}_L[g] |^2)  - 2\E_{\mu_L} |\widetilde{\Psi}_L[g] | + 1 \le  \E_{\mu_L} ( |\widetilde{\Psi}_L[g] |^2)  -  1. 
\end{align}
Since $\varepsilon/M \le \inf_\R g \le \sup_\R g \le M/\varepsilon$,  by Lemma \ref{lem-tilde-2-norm}, there exists $C_{\varepsilon, M}>0$, such that
\begin{align*}
\E_{\mu_L} (|\widetilde{\Psi}_L[g] |^2) &\le \exp\left( C_{\varepsilon, M} \int_\R | g(x) -1|^2 K_L(x,x)dx\right).
\end{align*}
Hence there exists $C_{\varepsilon, M}'>0$, such that 
\begin{align}\label{a-lem}
 \E_{\mu_L} (|\widetilde{\Psi}_L[g] |^2)  \le \exp\left( C_{\varepsilon, M}' \int_\R | g_1(x) -g_2(x)|^2 K_L(x,x)dx\right).
\end{align}
Substituting the inequalities \eqref{a-jensen} and \eqref{a-lem} into \eqref{cs}, we obtain the desired inequality. 
\end{proof}

\subsection{Proof of Theorem \ref{thm-main-bis}}

\begin{proof}[Proof of Theorem \ref{thm-main-bis}]
By Theorem \ref{thm-sub-main}, it suffices to check the inequality \eqref{reg-condition} under the assumption of Theorem \ref{thm-sub-main}. Indeed, we have 
\begin{align*}
& \int_\R | g(x) - 1|^2 K_L(x,x)dx 
\\
 = & \int_{|x|\ge \varepsilon} | g(x) - 1|^2 K_L(x,x)dx  + \int_{|x|< \varepsilon} | g(x) - 1|^2 K_L(x,x)dx
\\ 
 =&: I + II.
\end{align*}
The relation $I < \infty$ follows from the boundedness of $g$ and the assumption \eqref{off-s}. For the second term, we have
\begin{align*}
II \le C^2 \int_{| x| < \varepsilon} x^2 K_L(x,x)dx < \infty.
\end{align*}
This  proof of Theorem \ref{thm-main-bis} is complete.
\end{proof}

\subsection{Proof of Theorem B}
By \cite[Thm. 2.4]{PV} and \cite[\S 6]{BO-hyper}, if we assume that 
$$
\left|\frac{z+z'}{2} \right|< \frac{1}{2},
$$
then the Whittaker kernel $\mathcal{K} = \mathcal{K}_{z, z'}$ admits a bounded $L$-operator as in \eqref{L-integrable}, such that the bounded operator $V: L^2(\R_{-}) \rightarrow L^2(\R_{+})$ has as kernel:
$$
\frac{\sin \pi z \sin \pi z'}{\pi^2}\frac{\Big(\frac{x}{-y}\Big)^{\frac{z+z'}{2}}e^{-\frac{x-y}{2}} }{x-y},   \text{\, where \,} x>0, y < 0. 
$$ 
In other words, the $L$-kernel $ \mathcal{L}(x,y) = \mathcal{L}_{z,z'} (x,y)$ of the kernel $\mathcal{K}(x,y) = \mathcal{K}_{z,z'}(x,y)$ is given by
\begin{align}\label{W-L}
\mathcal{L}_{z,z'}(x,y) =  \frac{\mathcal{A}^{+}(x) \mathcal{A}^{-} (y) + \mathcal{A}^{-} (x) \mathcal{A}^{+}(y)}{x-y} ,\, x, y \in \R^*,
\end{align}
where 
$$
\mathcal{A}(x)  = \frac{ \sqrt{\sin \pi z \sin \pi z'}}{\pi}   |x|^{\sgn(x) \frac{z+z'}{2}}   e^{-\frac{|x|}{2}},  \text{\, where $x\ne 0$.}
$$
This function $\mathcal{A}$ satisfies the following conditions:
\begin{itemize}
\item the support of  $\mathcal{A}$ in $\R^*$ is the whole punctured line $\R^*$;
\item $\mathcal{A}\in C^\infty(\R^*) \cap L^2(\R)$.
\end{itemize}

Thus we have shown the following 
\begin{lem}
If $\left|z+z' \right|< 1$, then the  $L$-kernel $\mathcal{L} = \mathcal{L}_{z, z'}$ in \eqref{W-L} satisfies Condition \ref{con35}.
\end{lem}

Recall that if  $\mathfrak{p} = (p^{+}_1, \dots p^{+}_n;  p^{-}_1, \dots, p^{-}_n)$, then we set 
$$
g_{\mathfrak{p}}(x) = \prod_{i=1}^n \left(\frac{x-p_i^{+}}{x-p_i^{-}} \chi_{\{x > 0\}} + \frac{x-p_i^{-}}{x-p_i^{+}} \chi_{\{x<0\}}\right).
$$ 
Let $\lambda : = \frac{ |p_1^{-} \cdots p_n^{-}| }{p_1^{+} \cdots p_n^{+}}$ and recall the formula \eqref{c-form}: $c_\lambda =  \lambda \chi_{\R_{+}} + \lambda^{-1} \chi_{\R_{-}}$. Set
$$
h_{\mathfrak{p}} (x) = c_\lambda(x) |g_{\mathfrak{p}} (x) |  = c_\lambda(x) \sgn(g_{\mathfrak{p}}(x))  g_{\mathfrak{p}}(x).
$$
That is, 
$$
h_{\mathfrak{p}}(x) = \prod_{i=1}^n \left|\frac{x/p_i^{+}-1}{x/p_i^{-}-1} \chi_{\{x > 0\}} + \frac{x/p_i^{-}-1}{x/p_i^{+}-1} \chi_{\{x<0\}}\right|.
$$

The proof of the following lemma is immediate.
\begin{lem}\label{lem-hp}
The function $h_{\mathfrak{p}}$ is bounded and there exists $C > 0$ such that 
\begin{align*}
|h_{\mathfrak{p}}(x) - 1| \le C |x|.
\end{align*}
Moreover, for any $\varepsilon >0$, the subset $ \{ x \in \R^*: | h_{\mathfrak{p}}(x)  - 1| > \varepsilon\}$ is away from $0$.
\end{lem}

\begin{proof}[Proof of Proposition\ref{intro-prop-1} and Theorem B]
By Proposition \ref{prop-palm-kernel}, we have $\mathscr{P}_{z,z'}^{\mathfrak{p}} = \mu_{g_{\mathfrak{p}} \mathcal{L}g_{\mathfrak{p}}}.$ By Lemma \ref{lem-conjugation}, we have 
$\mu_{g_{\mathfrak{p}} \mathcal{L} g_{\mathfrak{p}}} =\mu_{|g_{\mathfrak{p}}| \mathcal{L} |g_{\mathfrak{p}}|}  =  \mu_{h_{\mathfrak{p}} \mathcal{L} h_{\mathfrak{p}}}.
$
Finally,  by Theorem \ref{thm-main-bis} and Lemma \ref{lem-hp}, the following limit 
$$
\overline{S}_{\mathcal{L}}[\log h_{\mathfrak{p}}^2](X) : =\lim_{\delta\to 0^{+}}    \left( \sum_{x \in X, | x| \ge \delta}\log h_{\mathfrak{p}}(x)^2 - \E_{\mu_{\mathcal{L}}} \sum_{x \in X, | x| \ge \delta}\log h_{\mathfrak{p}}(x)^2 \right)
$$
exists for $\mu_{\mathcal{L}}$-almost every configuration $X\in \Conf(\R^*)$.  Moreover, the function 
$$
X \mapsto \exp (\overline{S}_{\mathcal{L}}[\log  h_{\mathfrak{p}}^2] (X))
$$
 is in $L^1(\Conf(\R^*), \mu_{\mathcal{L}})$ and we have 
$$
 \mu_{ h_{\mathfrak{p}} \mathcal{L}  h_{\mathfrak{p}} }(dX)=  \frac{\exp (\overline{S}_{\mathcal{L}}[\log  h_{\mathfrak{p}}^2] (X)  )}{\E_{\mu_{\mathcal{L}}} \Big[ \exp (\overline{S}_{\mathcal{L}}[\log  h_{\mathfrak{p}}^2]) \Big]} \cdot \mu_{\mathcal{L}}(dX),
$$
that is, 
$$
\mathscr{P}_{z,z'}^{\mathfrak{p}}(dX)=  \frac{\exp (\overline{S}_{\mathcal{L}}[\log  h_{\mathfrak{p}}^2] (X)  )}{\E_{\mu_{\mathcal{L}}} \Big[ \exp (\overline{S}_{\mathcal{L}}[\log  h_{\mathfrak{p}}^2]) \Big]} \cdot \mathscr{P}_{z,z'}(dX).
$$
\end{proof}

\section{Appendix}

\begin{proof}[Proof of Proposition \ref{prop-M}]
By homogenity, we may assume, without loss of generality, that $\| A\|_{\mathscr{L}_{1|2}} \le 1$ and $\| B\|_{\mathscr{L}_{1|2}} \le 1$. Write $A$ and $B$ in block forms:  
$$
A = \left[ \begin{array}{cc}  a_1 & b_1 \\ c_1 & d_1 \end{array}\right], \quad B = \left[ \begin{array}{cc}  a_2 & b_2 \\ c_2 & d_2 \end{array}\right],
$$
then we have 
$$
AB = \left[ \begin{array}{cc}  a_1a_2 + b_1 c_2 & a_1 b_2 + b_1d_2 \\ c_1 a_2 + d_1 c_2 & c_1 b_2 + d_1d_2 \end{array}\right].
$$
By applying the operator ideal property $\|a b\|_1 \le \| a\|_1 \| b\|$, $\|a b\|_1 \le \| a\|_2 \| b\|$ and the H\"older inequality $\| ab\|_1 \le \|a \|_2 \| b\|_2$,  we get
\begin{align*}
\| AB\|_{\mathscr{L}_{1|2}}  =  &   \| a_1a_2 + b_1 c_2 \|_1 + \| c_1 b_2 + d_1d_2\|_1 + \| a_1 b_2 + b_1d_2\|_2 + \| c_1 a_2 + d_1 c_2 \|_2
\\
\le &  \| a_1\|_1 \| a_2\|+ \|b_1\|_2 \| c_2\|_2  + \| c_1\|_2\| b_2\|_2  + \|d_1\|_1\| d_2\| 
\\
 & + \| a_1\|_1\| b_2\| + \| b_1\|_2 \|d_2\| + \| c_1\|_2 \|a_2\| +  \| d_1\| \| c_2\|_2
 \\
  \le  &  \| a_1\|_1 + \|b_1\|_2   + \| c_1\|_2   + \|d_1\|_1  + \| a_1\|_1 + \| b_1\|_2  + \| c_1\|_2  +  \| d_1\| 
  \\
  \le & 2 (   \| a_1\|_1 + \|b_1\|_2   + \| c_1\|_2   + \|d_1\|_1)  \le 2. 
\end{align*}
\end{proof}

\begin{proof}[Proof of Proposition \ref{prop-bdd-M}]
The proof is easy from the definition of $\mathscr{L}_{1|2}(L^2(\R))$ and the ideal property of trace-class and Hilbert-Schmidt class.
\end{proof}

\begin{proof}[Proof of Proposition \ref{prop-ideal}]
By the relation \eqref{inclusions}, under the hypothesis of Proposition \ref{prop-ideal} on $A, B$, the two operators $A, B$ are both in $\mathscr{L}_2(L^2(\R))$, hence $AB\in \mathscr{L}_1(L^2(\R))$. By the ideal property of $\mathscr{L}_1(L^2(\R))$, the operator $(1 + A)^{-1}AB$ belongs to $\mathscr{L}_1(L^2(\R))$ and hence belongs to $\mathscr{L}_{1|2}(L^2(\R))$. 
We can write
$$
(1 + A)^{-1} B = (1 + A)^{-1} ( (1 + A) B - AB)   = B - (1 + A)^{-1}AB,
$$
hence the operator  $(1 + A)^{-1} B$ belongs to $\mathscr{L}_{1|2}(L^2(\R))$. Similar argument yields the fact that the operator  $B (1 + A)^{-1}$ also belongs to $\mathscr{L}_{1|2}(L^2(\R))$.
\end{proof}

\begin{proof}[Proof of Proposition \ref{prop-Fm}]
Fix a pair of operators $A, B$ in $\mathscr{L}_{1|2}(L^2(\R))$.  Note first that by Proposition \ref{prop-M}, the operator $A+ B + AB$ is in the space $\mathscr{L}_{1|2}(L^2(\R))$, hence the extended Fredholm determinant $\det((1 + A)(1+B)) = \det(1+A+B+AB)$ is well-defined. By the multiplicativity property of the usual Fredholm determinant, the desired identity holds whenever $A, B \in \mathscr{L}_{1}(L^2(\R))$, see, e.g. \cite[Thm. 3.8]{Simon-det}. Thus by the continuity of the function $A\mapsto \det(1+A)$ on $\mathscr{L}_{1|2}(L^2(\R))$, for proving the desired identity, it suffices to show that there exist two sequences $(A_n)_{n\in \N}$ and $(B_n)_{n\in \N}$ in $ \mathscr{L}_{1}(L^2(\R))$ such that we have the following convergences in the space $\mathscr{L}_{1|2}(L^2(\R))$:
\begin{align}\label{condition-A-B}
A_n \xrightarrow{n\to \infty} A, \, B_n \xrightarrow{n\to \infty} B \an A_nB_n \xrightarrow{n\to \infty} AB.
\end{align}
To this end, take any two sequences $(P_n)_{n\in \N}$ and $(Q_n)_{n\in \N}$ of finite rank orthogonal projections on $L^2(\R_{+})$ and $L^2(\R_{-})$ respectively, assume that $P_n$ and $Q_n$ converge in the strong operator topology to the orthogonal projections $P_{+}$ and $P_{-}$ respectively. Now we may set $$A_n = (P_n + Q_n)A, \quad B = B(P_n + Q_n).$$
   Then it is clear that the finite rank operators $A_n$ and $B_n$ satisfy all the desired conditions in \eqref{condition-A-B}. Note that we intentionally obtain $A_n$ and $B_n$ by multiplying  $P_n + Q_n$ on the left side of $A$ and on the right side of $B$, so that the third condition in \eqref{condition-A-B} is satisfied. 
\end{proof}

\begin{proof}[Proof of Proposition \ref{prop-Fc}]
From Grothendieck's definition of Fredholm determinant: 
$$
\det(1 + T) = \sum_{k=0}^\infty \tr(\wedge^k (T)), \quad T \in  \mathscr{L}_1(L^2(\R)),
$$
and the fact that,  once $A\in \mathscr{L}_1(L^2(\R))$ and $f$ is a bounded function, then 
$$
\tr(\wedge^k (fA)) = \tr( \wedge^k (M_f) \circ \wedge^k (A)) =  \tr(  \wedge^k (A) \circ \wedge^k (M_f)) = \tr(\wedge^k (Af)),
$$ 
we see that the identity \eqref{identity-Fc} holds when $A\in \mathscr{L}_1(L^2(\R))$. For $A\in \mathscr{L}_{1|2}(L^2(\R))$, we may argue similarly as in the proof of Proposition \ref{prop-Fm}.  See also \cite{Buf-inf} for the proof in more general case. 
\end{proof}

\section*{Acknowledgements}
The authors  are supported by A*MIDEX project (No. ANR-11-IDEX-0001-02), financed by Programme ``Investissements d'Avenir'' of the Government of the French Republic managed by the French National Research Agency (ANR). A.B. is also supported in part by the Grant MD-2859.2014.1 of the President of the Russian Federation, by the Programme ``Dynamical systems and mathematical control theory'' of the Presidium of the Russian Academy of Sciences,  by a subsidy granted to the HSE by the Government of the Russian Federation for the implementation of the Global Competitiveness Program and by the RFBR grant 13-01-12449. Y.Q. is partially supported by the  grant  346300 for IMPAN from the Simons Foundation and the matching 2015-2019 Polish MNiSW fund.

\def\cprime{$'$} \def\cydot{\leavevmode\raise.4ex\hbox{.}}

\end{document}